\documentclass[reqno]{amsart}

\usepackage{amsmath,amsthm,amstext,amssymb,bbm}
\usepackage{hyperref}
\usepackage{enumerate,enumitem}

\usepackage{todonotes} 

\makeatletter
\@namedef{subjclassname@2020}{%
  \textup{2020} Mathematics Subject Classification}
\makeatother

\theoremstyle{plain}
\newtheorem{theorem}{Theorem}[section]
\newtheorem{lemma}[theorem]{Lemma}

\newtheorem{corollary}[theorem]{Corollary}
\newtheorem{proposition}[theorem]{Proposition}
\newtheorem{claim}[theorem]{Claim}
\newtheorem*{claim*}{Claim}
\newtheorem{definition}[theorem]{Definition}
\theoremstyle{definition}
\newtheorem{remark}[theorem]{Remark}

\newcommand{\betrag}[1]{\vert{#1}\vert}

\newcommand{\dom}[1]{{{\rm{dom}}(#1)}}
\newcommand{\supp}[1]{{{\rm{sprt}}(#1)}}
\newcommand{\crit}[1]{{{\rm{crit}}\left({#1}\right)}}
\newcommand{\cof}[1]{{{\rm{cof}}(#1)}}

\newcommand{\ran}[1]{{{\rm{ran}}(#1)}}

\newcommand{\POT}[1]{{\wp}({#1})}
\newcommand{\map}[3]{{#1}:{#2}\longrightarrow{#3}}
\newcommand{\Map}[5]{{#1}:{#2}\longrightarrow{#3};~{#4}\longmapsto{#5}}
\newcommand{\pmap}[4]{{#1}:{#2}\xrightarrow{#4}{#3}}
\newcommand{\Set}[2]{\{{#1}~\vert~{#2}\}}
\newcommand{\seq}[2]{\langle{#1}~\vert~{#2}\rangle}

\newcommand{\anf}[1]{{\text{``}\hspace{0.3ex}{#1}\hspace{0.3ex}\text{''}}}

\newcommand{\Poti}[2]{{\wp}_{#2}(#1)}
\newcommand{\POTI}[2]{{\wp}^*_{#2}(#1)}

\newcommand{\NS}[1]{{{NS}}_{#1}}
\newcommand{\CLUB}[1]{{{Club}}_{#1}}
\newcommand{\HH}[1]{{\rm{H}}(#1)}

\newcommand{\Add}[2]{{\rm{Add}}({#1},{#2})}
\newcommand{\Col}[2]{{\rm{Col}}({#1},{#2})}

\newcommand{\Lim}{{\rm{Lim}}}
\newcommand{\Ord}{{\rm{Ord}}}
\newcommand{\LL}{{\rm{L}}}

\newcommand{\ZFC}{{\rm{ZFC}}}

\newcommand{\PFA}{{\rm{PFA}}}
\newcommand{\MM}{{\rm{MM}}}
\newcommand{\FA}{{\rm{FA}}}

\newcommand{\PPP}{{\mathbb{P}}}
\newcommand{\QQQ}{{\mathbb{Q}}}
\newcommand{\RRR}{{\mathbb{R}}}

\newcommand{\VV}{{\rm{V}}}

\newcommand{\calC}{\mathcal{C}}
\newcommand{\calD}{\mathcal{D}}
\newcommand{\calE}{\mathcal{E}}

\newcommand{\calS}{\mathcal{S}}
\newcommand{\calT}{\mathcal{T}}

 \newcommand{\CH}{{\rm{CH}}}
 \newcommand{\GCH}{{\rm{GCH}}}

 \newcommand{\calN}{\mathcal{N}}

\title{Forcing axioms and the complexity of non-stationary ideals}

\author{Sean Cox}
\address{
Department of Mathematics and Applied Mathematics \\
Virginia Commonwealth University \\
1015 Floyd Avenue \\
Richmond, Virginia 23284, USA.  
}

\author{Philipp L\"ucke}
\address{Facultat de Matem\`{a}tiques i Inform\`{a}tica, Universitat de Barcelona. 
Gran via de les Corts Catalanes 585,
08007 Barcelona, Spain.
}

\thanks{This project has received funding from the European Union’s Horizon 2020 research and innovation programme under the Marie Sk{\l}odowska-Curie grant agreement No 842082 of the second author (Project \emph{SAIFIA: Strong Axioms of Infinity -- Frameworks, Interactions and Applications}), and from the Simons Foundation Grant number 318467 of the first author.  
The authors would like to thank the anonymous referee for the detailed reading of the manuscript and several helpful comments.}

\subjclass[2020]{03E57, 03E35, 03E47}
\keywords{Non-stationary ideals, $\Delta_1$-definability, Martin's Maximum, Subcomplete Forcing Axiom, stationary reflection}

\begin{document}

\begin{abstract}
  We study the influence of strong forcing axioms on the complexity of the non-stationary ideal on $\omega_2$ and its restrictions to certain cofinalities. 
  Our main result shows that the strengthening $\MM^{++}$ of Martin's Maximum does not decide whether the restriction of the non-stationary ideal on $\omega_2$ to sets of ordinals of countable cofinality is $\Delta_1$-definable by formulas with parameters in $\HH{\omega_3}$. 
  The techniques developed in the proof of this result also allow us to prove analogous results for the full non-stationary ideal on $\omega_2$ and strong forcing axioms that are compatible with $\CH$. 
  Finally, we answer a question of S. Friedman, Wu and Zdomskyy  by showing that the $\Delta_1$-definability of the non-stationary ideal on $\omega_2$ is compatible with arbitrary large values of the continuum function at $\omega_2$. 
\end{abstract}

\maketitle



\section{Introduction}

  The fact that closed unbounded subsets generate a proper normal filter, the \emph{club filter} on $\kappa$ $$\CLUB{\kappa} ~ = ~ \Set{A\subseteq \kappa}{\textit{$\exists C\subseteq A$ closed and unbounded in $\kappa$}},$$  is one of the most important combinatorial properties of uncountable regular cardinals $\kappa$.  
 The study of the structural properties of these filters and their dual ideals, the \emph{non-stationary ideal} on $\kappa$ $$\NS{\kappa} ~ = ~ \Set{A\subseteq \kappa}{\textit{$\exists C$ closed and unbounded in $\kappa$ with $A\cap C=\emptyset$}}$$ plays a central role in modern set theory.

 In \cite{MR1222536} and \cite{MR1242054}, Mekler, Shelah and V\"a\"an\"anen initiated the study of the \emph{complexity} of club filters and  non-stationary ideals, leading to various results establishing interesting connections between the complexity of these objects and their structural properties. 
 Given an uncountable regular cardinal $\kappa$, it is easy to see that both $\CLUB{\kappa}$ and $\NS{\kappa}$ are definable by a $\Sigma_1$-formula with parameter $\kappa$, i.e. there exist $\Sigma_1$-formulas $\varphi_0(v_0,v_1)$ and $\varphi_1(v_0,v_1)$ such that $\CLUB{\kappa}=\Set{A}{\varphi_0(A,\kappa)}$ and $\NS{\kappa}=\Set{A}{\varphi_1(A,\kappa)}$. 
 The results of \cite{MR1242054}  show that under $\CH$, the $\mathbf{\Delta}_1(\HH{\omega_2})$-definability of $\NS{\omega_1}$ (i.e. the assumption that $\NS{\omega_1}=\Set{A}{\psi_1(A,z)}$ holds for some $\Pi_1$-formula $\psi(v_0,v_1)$ and some $z\in\HH{\omega_2}$) is equivalent to several interesting combinatorial and model-theoretic assumptions about objects of size $\omega_1$. 
  In particular, it is shown that this definability assumption is equivalent to  the existence of a so-called \emph{canary tree}, a tree of height and cardinality $\omega_1$ without cofinal branches that has specific properties with respect to the ordering of such trees under order-preserving embeddings. 
 Since the results of \cite{MR1222536} show that the existence of a canary tree is independent of $\ZFC+\CH$, it follows that this theory is not able to determine the exact complexity of $\NS{\omega_1}$.

 The above results were later generalized to higher cardinals. 
  If $S$ is a stationary subset of an uncountable regular cardinal $\kappa$, then we let $\NS{}\restriction S=\NS{\delta}\cap\POT{S}$ denote the \emph{restriction} of the non-stationary ideal on $\delta$ to $S$. Given infinite regular cardinals $\lambda<\kappa$, we set $S^\kappa_\lambda=\Set{\alpha<\kappa}{\cof{\alpha}=\lambda}$. In addition, if $m<n<\omega$, then we write $S^n_m$ instead of $S^{\omega_n}_{\omega_m}$. 
 Results of Hyttinen and Rautila in \cite{MR1877015}  showed that if $\kappa$ is an infinite regular cardinal in a model of the $\GCH$, then, in a cofinality-preserving forcing extension, the set $\NS{}\restriction S^{\kappa^+}_\kappa$ is $\mathbf{\Delta}_1(\HH{\kappa^{++}})$-definable. 
 Furthermore, in \cite{MR3320477}, S. Friedman, Wu and Zdomskyy showed that for every successor cardinal in G\"odel's constructible universe $\LL$, there is a cardinality-preserving forcing extension of $\LL$ in which $\NS{\kappa}$ is $\mathbf{\Delta}_1(\HH{\kappa^+})$-definable. 
 These results can be easily used to show that the complexity of the non-stationary ideal  and its restriction is not determined by $\ZFC$ (see Lemma \ref{lemma:Negative_results} and the subsequent discussion below). 
 Finally, recent work also unveiled several interesting consequences of the $\mathbf{\Delta}_1(\HH{\kappa^+})$-definability of restriction of $\NS{\kappa}$ at higher cardinals $\kappa$. 
 In particular, this set-theoretic assumption was shown to be closely connected to model-theoretic questions dealing with Shelah's \emph{Classification Theory} and the complexity of certain  mathematical theories (see, for example, {\cite[Theorem 64]{MR3235820}}).

 The above results strongly motivate the question whether canonical extensions of $\ZFC$ decide more about the complexity of non-stationary ideals, and this question turns out to be closely connected to important recent developments in set theory. 
 In \cite{MR3235820}, S. Friedman, Hyttinnen and Kulikov showed that, in the constructible universe $\LL$, the sets of the form $\NS{}\restriction S$ for some stationary subset $S$ of an uncountable regular cardinal $\kappa$ are not $\mathbf{\Delta}_1(\HH{\kappa^+})$-definable. 
 Using the notion of \emph{local club condensation} (see \cite{MR2860182}), it is possible to extend this conclusion to larger canonical inner models. 
  In another direction, S. Friedman and Wu observed in \cite{FriedmanWu} that strong saturation properties of the non-stationary ideal on $\omega_1$, i.e. the assumption that the poset $\POT{\omega_1}/\NS{\omega_1}$ has a dense subset of cardinality $\omega_1$, imply the $\mathbf{\Delta}_1(\HH{\omega_2})$-definability of $\NS{\omega_1}$. Results of Woodin in {\cite[Chapter 6]{MR2723878}} show that $\NS{\omega_1}$ possesses these properties in certain forcing extensions of determinacy models. 
   Finally, Schindler and his collaborators recently studied the question whether forcing axioms determine the complexity of $\NS{\omega_1}$. In \cite{LarsonSchindlerWu}, Larson, Schindler and Wu showed that \emph{Woodin's Axiom $(*)$} (see {\cite[Definition 5.1]{MR2723878}}) implies that $\NS{\omega_1}$ is not $\mathbf{\Delta}_1(\HH{\omega_2})$-definable. 
 In combination with recent results of Asper\'o and Schindler in \cite{AsperoSchindler}, this shows that $\MM^{++}$, a natural strengthening of \emph{Martin's Maximum}, implies that  $\NS{\omega_1}$ is not $\mathbf{\Delta}_1(\HH{\omega_2})$-definable.

  The work presented in this paper is motivated by the question whether strong forcing axioms determine the complexity of the non-stationary ideal on $\omega_2$ and its restrictions. 
The following result from \cite{MR3952233} shows that all extensions of $\ZFC$ that are preserved by forcing with ${<}\omega_2$-directed posets are compatible with the assumption that for every stationary subset $S$ of $\omega_2$, the set $\NS{}\restriction S$ is not $\mathbf{\Delta}_1(\HH{\omega_3})$-definable. 
 In particular, the results of \cite{cox2018forcing}, \cite{doi:10.1002/malq.200410101} and \cite{MR1782117} show that this statement is compatible with all standard forcing axioms, like $\MM^{++}$. 
 The lemma follows directly from a combination of {\cite[Theorem 2.1]{MR3952233}}, showing that no $\mathbf{\Delta}^1_1$-definable set (see {\cite[Definition 1.2]{MR2987148}}) separates $\CLUB{\kappa}$ from $\NS{\kappa}$ in the given model of set theory, and {\cite[Lemma 2.4]{MR2987148}}, showing that $\mathbf{\Delta}^1_1$-definability coincides with $\mathbf{\Delta}_1(\HH{\kappa^+})$-definability at all uncountable regular cardinals $\kappa$.

\begin{lemma}\label{lemma:Negative_results}
 Let $\kappa$ be an uncountable cardinal with $\kappa^{{<}\kappa}=\kappa$ and let $G$ be $\Add{\kappa}{\kappa^+}$-generic over $\VV$. In $\VV[G]$, no $\mathbf{\Delta}_1(\HH{\kappa^+})$-definable subset of $\POT{\kappa}$ separates $\CLUB{\kappa}$ from $\NS{\kappa}$, i.e. no set $X$ definable in this way satisfies $\CLUB{\kappa}\subseteq X\subseteq\POT{\kappa}\setminus\NS{\kappa}$. 
 \end{lemma}

Note that, if $S$ is a stationary subset of an uncountable regular cardinal $\kappa$, then $\NS{}\restriction S$ separates $\CLUB{\kappa}$ from $\NS{\kappa}$. This shows that, in $\Add{\kappa}{\kappa^+}$-generic extensions, sets of the form $\NS{}\restriction S$ for stationary subsets $S$ are not $\mathbf{\Delta}_1(\HH{\kappa^+})$-definable.

 In contrast, we will prove the following theorem that shows that strong forcing axioms like $\MM^{++}$ are also compatible with the existence of a $\mathbf{\Delta}_1(\HH{\omega_3})$-definable set that separates the club filter on $\omega_2$ from the corresponding non-stationary ideal. 
 The proof of this result is based on a detailed analysis of the preservation properties of a variation of a forcing iteration constructed by Hyttinen and Rautila in the consistency proofs of  \cite{MR1877015}. 
 Our construction will also allow us to produce such models with arbitrary large $2^{\omega_2}$.  See Section \ref{section:Prelim} for the meaning of the ``$+\mu$" versions of forcing axioms.\footnote{In keeping with the prevailing convention in the literature: $\text{MM}^+$ refers to $\text{MM}^{+1}$, but $\text{MM}^{++}$ refers to $\text{MM}^{+\omega_1}$, \emph{not} to $\text{MM}^{+2}$ (and similarly for PFA and other forcing axioms).}

 \begin{theorem}\label{MainTheorem1}
Let $\FA$ denote any one of the following forcing axioms:
\begin{itemize}
 \item   $\MM^{\hspace{1.2pt}{+}\mu}$, where $\mu$ is a cardinal and $0 \le \mu \le \omega_1$; or
 \item $\PFA^{{+}\mu}$, where $\mu$ is a cardinal and $1 \le \mu \le \omega_1$.
\end{itemize} 
 Assume that $\FA$ holds, and let $\theta$ be a cardinal with $\theta^{\omega_2}=\theta$. 
  Then there exists a ${<}\omega_2$-directed closed, cardinal-preserving poset $\PPP$ with the property that whenever $G$ is $\PPP$-generic over $\VV$, then, in $\VV[G]$, the axiom $\FA$ still holds, $2^{\omega_2}=\theta$ and the set $\NS{}\restriction S^2_0$ is $\mathbf{\Delta}_1(\HH{\omega_3})$-definable. 
 \end{theorem}

We will also apply the techniques developed in the proof of the above result to forcing axioms that are compatible with the continuum hypothesis, focusing on the axiom $\FA^{{+}}(\text{$\sigma$-closed})$ and the \emph{subcomplete forcing axiom} $\mathrm{SCFA}$ introduced by Jensen in \cite{MR2840749}.  Note that both axioms are preserved by ${<}\omega_2$-directed closed forcings (see \cite{cox2018forcing} and \cite{MR1782117}) and hence Lemma \ref{lemma:Negative_results} above already shows that they are compatible with the assumption that no $\mathbf{\Delta}_1(\HH{\omega_3})$-definable set separates $\CLUB{\omega_2}$ from $\NS{\omega_2}$.

 \begin{theorem}\label{MainTheorem2}
  Let $\FA$ denote either the axiom $\mathrm{SCFA}$ or the axiom $\FA^{{+}}(\text{$\sigma$-closed})$. 
  Assume that $2^\omega=\omega_1$, $2^{\omega_1}=\omega_2$ and $\FA$ holds.  Let $\theta$ be a cardinal satisfying $\theta^{\omega_2}=\theta$. 
   Then there exists a ${<}\omega_2$-directed closed, cardinal-preserving poset $\PPP$ with the property that whenever $G$ is $\PPP$-generic over $\VV$, then, in $\VV[G]$, the axiom $\FA$ still holds,   $2^{\omega_2}=\theta$ and the set $\NS{\omega_2}$ is $\mathbf{\Delta}_1(\HH{\omega_3})$-definable. 
 \end{theorem}

 This theorem also provides an affirmative answer to {\cite[Problem 3.3]{MR3320477}} posed by S. Friedman, Wu and Zdomskyy, by showing that the $\mathbf{\Delta}_1$-definability of $\NS{\omega_2}$ is compatible with $2^{\omega_2}\geq\omega_5$.


\section{Preliminaries}\label{section:Prelim}

This section covers well-known results, mostly related to the notions of internal approachability and Shelah's Approachability Ideal.

  First we recall the ``plus" versions of forcing axioms, which were first introduced by Baumgartner~\cite{MR776640} (though the prevailing notation has changed somewhat since then).  If $\Gamma$ is a class of posets and $\mu$ is a cardinal, $\text{FA}^{+\mu}\big(\Gamma \big)$ states that for every poset $\PPP \in \Gamma$, for every collection $\calD$ of size $\omega_1$ of dense subsets of $\PPP$ and for every sequence $\seq{\sigma_\xi}{\xi<\mu}$ of $\PPP$-names for stationary subsets of $\omega_1$, there is a $\calD$-generic filter $g$ on $\PPP$ with the property that the set $\Set{\alpha<\omega_1}{\exists p\in g ~ p\Vdash_\PPP\anf{\check{\alpha}\in\sigma_\xi}}$ is stationary in $\omega_1$ for every $\xi<\mu$.  If $\Gamma$ is the class of posets that preserve stationary subsets of $\omega_1$, we write $\MM^{+\mu}$ instead of $\text{FA}^{+\mu}(\Gamma)$; and, as mentioned earlier, $\MM^{++}$ refers to $\MM^{+\omega_1}$, \emph{not} to $\MM^{+2}$.  Similar comments apply to the class of proper posets and PFA.

\begin{definition}\label{def_GenericAndTotalMaster}
 Let $\PPP$ be a poset and let $W \prec (\HH{\theta},\in,\PPP)$ for some sufficiently large regular cardinal $\theta$.
 \begin{enumerate}
 \item A condition $p \in \PPP$ is a $\boldsymbol{(W,\PPP)}$\textbf{-master condition} if $$W[G]\cap\VV ~ = ~ W$$ holds whenever $G$ is $\PPP$-generic over $\VV$ with $p\in G$. 
  
  \item A set $g$ is $\boldsymbol{(W,\PPP)}$\textbf{-generic} if $g \subseteq\PPP\cap W$, $g$ is a filter on $\PPP\cap W$, and $D\cap g  \neq \emptyset$ for every $D \in W$ that is a dense subset of $\PPP$.\footnote{Sometimes the requirement that $g \subseteq W$ is dropped, but then one has the demand that $D\cap g  \cap W\neq\emptyset$ holds for each dense $D \in W$.}  
  
  \item A condition $p \in \mathbb{P}$ is a $\boldsymbol{(W,\PPP)}$\textbf{-total master condition} if the set $$\Set{r \in \PPP\cap W}{p\leq_\PPP r}$$ is a $(W,\PPP)$-generic filter. 
 \end{enumerate} 
\end{definition}

The following result is well-known:

\begin{lemma}\label{lem_TotalMasterDist}
  Let $\PPP$ be a poset, let $W \prec (\HH{\theta},\in,\PPP)$, let $\mu$ be an ordinal with $\mu\subseteq W$, and let $\dot{f} \in W$ be a $\PPP$-name for a function from $\mu$ to the ground model $\VV$. 
   If $p$ is a $(W,\PPP)$-total master condition and $G$ is $\PPP$-generic over $\VV$ with $p\in G$, then $\dot{f}^G\in\VV$.  
\end{lemma}

\begin{proof}
 Fix a $(W,\PPP)$-total master condition $p$ and a filter $G$ on $\PPP$ that is generic over $\VV$ and contains the condition $p$. Let $g=\Set{r \in\PPP\cap W}{p\leq_\PPP r}$ denote the $(W,\PPP)$-generic filter induced by $p$.  Then $G \cap W$ is a $(W,\PPP)$-generic filter extending $g$ and therefore standard arguments show that $G \cap W=g$. 
 By elementarity, there is a sequence $\seq{A_\xi}{\xi<\mu}\in W$ of maximal antichains in $\PPP$ with the property that for every $\xi<\mu$, each condition in $A_\xi$ decides the value of $\dot{f}$ at $\xi$.  
  Since $\mu\subseteq W$, this shows that for all $\xi<\mu$, the unique condition in $A_\xi\cap g$ decides the value of $\dot{f}$ at $\xi$. But the sequence $\seq{A_\xi}{\xi<\mu}$ and the filter $g$ are both elements of $\VV$ and hence the function $\dot{f}^G$ is also in the ground model.  
 \end{proof}

 We state a definition that will be used extensively in the following arguments:

\begin{definition}\label{def_IA}
 Given an infinite regular  cardinal $\kappa$, we let $\text{IA}_\kappa$ denote the class of all sets $W$ with the property that there exists a sequence $\vec{N} = \seq{N_\alpha}{\alpha < \kappa}$ that satisfies the following statements: 
 \begin{enumerate}
   \item The sequence $\vec{N}$ is $\subseteq$-increasing and $\subseteq$-continuous. 
  
  \item $W = \bigcup\Set{N_\alpha}{\alpha < \kappa}$. 
 
  \item $\betrag{N_\alpha} < \kappa$ for all $\alpha<\kappa$. 
 
  \item Every proper initial segment of $\vec{N}$ is an element of $W$. 
 \end{enumerate} 
\end{definition}

\begin{remark}
If $\vec{N}$ witnesses that $W$ is an element of $\text{IA}_\kappa$ and $W \prec \HH{\theta}$ for some $\theta > \kappa$, then $\kappa \subseteq W$. This is because we have $\vec{N} \restriction\alpha \in W$ for every $\alpha < \kappa$, and the domain of $\vec{N} \restriction \alpha$, namely $\alpha$, is definable from the parameter $\vec{N} \restriction \alpha$. 
\end{remark}

In what follows, if $\tau$ is a regular uncountable cardinal, $\Poti{H}{\tau}$ refers to the set of all $W \subseteq H$ with $\betrag{W}<\tau$, and $\POTI{H}{\tau}$ denotes the set $$\Set{W \in\Poti{H}{\tau}}{W \cap \tau \in \tau}.$$ The set $\POTI{\HH{\theta}}{\tau}$ contains a club in the sense of Jech (see \cite{JECH1973165}), but not necessarily in the sense of Shelah (see \cite{10.2307/1971415}).

\begin{remark}
 In the above situation, if $W \in \POTI{\HH{\theta}}{\tau}$, $W \prec \HH{\theta}$, and $x \in W$ with $\betrag{x} < \tau$, then $x \subseteq W$.\footnote{Note that this could fail if $W$ were allowed to have non-transitive intersection with $\tau$.}
\end{remark}

\begin{lemma}\label{lem_IA_is_stat}
 If $\kappa$ is a regular and uncountable cardinal, then $\text{IA}_\kappa$ is stationary in $\Poti{\HH{\theta}}{\kappa^+}$ for all sufficiently large regular $\theta$. 
\end{lemma}

\begin{proof}
Given a first-order structure $\mathfrak{A} = (\HH{\theta}, \in, \kappa, \ldots)$ in a countable language, recursively construct a $\subseteq$-continuous and $\subseteq$-increasing sequence $\vec{N} = \seq{N_\alpha}{\alpha< \kappa}$ of elementary substructures of $\mathfrak{A}$ of cardinality less than $\kappa$ such that $\vec{N} \restriction \alpha \in N_{\alpha+1}$ for all $\alpha < \kappa$.  Then $\vec{N}$ witnesses that its union is contained in $\text{IA}_\kappa$.
\end{proof}

\begin{lemma}\label{lem_ClosureYieldsGenerics2}
 Let $\PPP$ be a poset, let $\kappa<\theta$ be infinite regular cardinals with $\PPP\in\HH{\theta}$, let $\lhd$ be a well-ordering of $\HH{\theta}$, let $W \prec (\HH{\theta},\in,\PPP,\lhd)$ with $W \in \text{IA}_\kappa$, and let $p \in\PPP\cap W$.  
 \begin{enumerate} 
  \item If $\PPP$ is ${<}\kappa$-closed, then there exists a $(W,\PPP)$-generic filter that contains $p$.  

  \item If $\PPP$ is ${<}\kappa^+$-closed, then there exists a $(W,\PPP)$-total master condition below $p$.
\end{enumerate}
\end{lemma}

\begin{proof}
 Let $\vec{N} = \seq{N_\alpha}{\alpha < \kappa}$ witness that $W$ is an element of $\text{IA}_\kappa$. 
 
 (1)  Assuming that $\PPP$ is ${<}\kappa$-closed. Using the closure of $\PPP$ and the fact that each $N_\alpha$ has cardinality less than $\kappa$, we can recursively construct a descending sequence $\vec{p}=\seq{p_\alpha}{\alpha < \kappa}$ of conditions below $p$ in $\PPP$ such that the following statements hold for all $\alpha< \kappa$:
  \begin{enumerate}
   \item[(a)] The condition $p_{\alpha+1}$ is the $\lhd$-least element of $\PPP$ below $p_\alpha$ that is an element of every open dense set that belongs to $N_\alpha$.\footnote{We do not require here that $p_{\alpha+1}$ is a total master condition for $N_\alpha$.  That is, if $D \in N_\alpha$ is dense, the upward closure of $p_{\alpha+1}$ is only required to meet $D$, not necessarily $D \cap N_\alpha$.}  
   
   \item[(b)] If $\alpha$ is a limit ordinal, then $p_\alpha$ is the $\lhd$-least lower bound of the sequence $\seq{p_\ell}{\ell < \alpha}$.  
  \end{enumerate} 
 Then every proper initial segment of $\vec{p}$ is definable from a proper initial segment of $\vec{N}$, and hence every proper initial segment of $\vec{p}$ is in $W$. In particular, we know that $p_{\alpha+1} \in W$ for all $\alpha < \kappa$.  It follows that the filter in $\PPP$ generated by the subset $\Set{p_\alpha}{\alpha < \kappa}$ is $(W,\PPP)$-generic. 

 (2) Now, assume that $\PPP$ is ${<}\kappa^+$-closed and repeat the above construction of the sequence $\vec{p}$. Then $\vec{p}$ has a lower bound in $\PPP$, and this lower bound is clearly a $(W,\PPP)$-total master condition.
\end{proof}

Next we discuss one variant of proper forcing.

\begin{definition}\label{def_ProperForIA}
Let $\kappa$ be an infinite regular cardinal. 
 \begin{enumerate}
   \item A poset $\PPP$ is \textbf{$\boldsymbol{IA}_{\boldsymbol{\kappa}}$-proper} if for all sufficiently large regular cardinals $\theta$, all $W \prec (\HH{\theta},\in,\PPP)$ with $W \in \text{IA}_\kappa$ and all $p \in\PPP\cap W$, 
    there is a $(W,\mathbb{P})$-master condition below $p$.

 \item A poset $\PPP$ is  \textbf{$\boldsymbol{IA}_{\boldsymbol{\kappa}}$-totally proper} if for all sufficiently large regular cardinals $\theta$, all $W \prec (\HH{\theta},\in,\PPP)$ with $W \in \text{IA}_\kappa$ and all $p \in\PPP\cap W$, there is a $(W,\mathbb{P})$-total master condition below $p$. 
 \end{enumerate}
\end{definition}

It is well-known that $\text{IA}_\kappa$-proper posets preserve all stationary subsets of $S^{\kappa^+}_\kappa$ that lie in the approachability ideal $I[\kappa^+]$ defined below. 
 Since we could not find a reference for exactly what is needed in our arguments, we sketch the proof below.  Note that it is possible for $\text{IA}_\kappa$-proper (even $\text{IA}_\kappa$-totally proper) posets to destroy the stationarity of some subsets of $S^{\kappa^+}_\kappa$ (see \cite{MR2160657}). 
  So $\text{IA}_\kappa$-total properness is, in general, strictly weaker than $\kappa^+$-Jensen completeness (defined in the next section), because ${<}\kappa^+$-closed forcings preserve all stationary subsets of $\kappa^+$.

\begin{definition}[Shelah]
 Let $\kappa$ be an infinite regular cardinal. 
 \begin{enumerate}
  \item Given a sequence $\vec{z} = \seq{z_\alpha}{\alpha < \kappa^+}$ a sequence of elements of $[\kappa^+]^{{<}\kappa}$, an ordinal $\gamma < \kappa^+$ is called \textbf{approachable with respect to $\boldsymbol{\vec{z}}$} if there exists a sequence $$\vec{\alpha} ~ = ~ \seq{\alpha_\xi}{\xi < \cof{\gamma}}$$ 
cofinal in $\gamma$ such that every proper initial segment of $\vec{\alpha}$ is equal to $z_\alpha$ for some $\alpha < \gamma$.  

  \item The \textbf{Approachability ideal} $\boldsymbol{I[\kappa^+]}$ on $\kappa^+$ is the (possibly non-proper) \emph{normal} ideal generated by sets of the form $$A_{\vec{z}} ~ = ~ \Set{\gamma < \kappa^+}{\textit{$\gamma$ is approachable with respect to $\vec{z}$ }}$$ for some sequence $\vec{z} \in {}^{\kappa^+}([\kappa^+]^{{<}\kappa})$. 
 \end{enumerate}
\end{definition}

 Note that a subset $X$ of $\kappa^+$ is an element of  $I[\kappa^+]$ if and only if there exists some club $D \subseteq \kappa^+$ and some sequence $\vec{z} \in {}^{\kappa^+}([\kappa^+]^{{<}\kappa})$ such that every $\gamma \in D \cap X$ is approachable with respect to $\vec{z}$. 
 In the following, we will make use of several facts about $I[\kappa^+]$.  Throughout this section, $\kappa$ denotes a regular cardinal.

\begin{lemma}[\cite{MR2160657}]\label{lem_CardArith_MaximalElement}
Suppose $\kappa^{{<}\kappa} \leq \kappa^+$, and let $\seq{z_
\alpha}{\alpha < \kappa^+}$ be an enumeration of $[\kappa^+]^{{<}\kappa}$.\footnote{Note that such an enumeration exists by our cardinal arithmetic assumption.}  Define $$M_{\vec{z}} ~ = ~ \Set{\gamma\in S^{\kappa^+}_\kappa}{\textit{$\gamma$ is approachable with respect to $\vec{z}$ }}.$$  Then the following statements hold: 
\begin{enumerate}
  \item $M_{\vec{z}}$ is a stationary subset of $S^{\kappa^+}_\kappa$. 
  
  \item $M_{\vec{z}} \in I[\kappa^+]$. 
  
  \item $M_{\vec{z}}$ is a \emph{maximum element of $I[\kappa^+]\cap\POT{S^{\kappa^+}_\kappa}$ mod NS}, i.e.  whenever $S$ is a stationary subset of $S^{\kappa^+}_\kappa$ such that $S \in I[\kappa^+]$, then $S \setminus M_{\vec{z}}$ is non-stationary. 
  
  \item\label{item_CardArith_AP_fail} If $\kappa^{{<}\kappa} = \kappa$, then $S^{\kappa^+}_\kappa \setminus M_{\vec{z}}$ is non-stationary. In particular, $\kappa^{{<}\kappa} = \kappa$ implies that $S^{\kappa^+}_\kappa \in I[\kappa^+]$. 
 \end{enumerate}
\end{lemma}

\begin{proof}
 (1)  Fix a sufficiently large regular cardinal $\theta$ and a well-ordering $\lhd$ of $\HH{\theta}$.  
 Fix $W \in \text{IA}_\kappa$ with $W \prec (\HH{\theta},\in,\lhd,\vec{z})$ and let $\seq{N_\alpha}{\alpha < \kappa}$ be a sequence witnessing that $W \in \text{IA}_\kappa$. Given $\alpha<\kappa$, set $\gamma_\alpha=\sup(N_\alpha \cap \kappa^+)<\kappa^+$. Then $\vec{\gamma}=\seq{\gamma_\alpha}{\alpha<\kappa}$ enumerates a cofinal subset of $W \cap \kappa^+$ of order-type $\kappa$ and every proper initial segment of this sequence is an element of $W$.  
 Moreover, each proper initial segment of $\vec{\gamma}$ is an element of $[\kappa^+]^{{<}\kappa}$, and hence an element of $$W \cap \Set{ z_\alpha}{\alpha < \kappa^+} ~ = ~  \Set{z_\alpha}{\alpha < W \cap \kappa^+}.$$ This shows that $W \cap \kappa^+$ is approachable with respect to $\vec{z}$. 
 Since Lemma \ref{lem_IA_is_stat} shows that  there are stationarily-many $W \in \text{IA}_\kappa$ with $W \prec (\HH{\theta},\in,\lhd,\vec{z})$, these computations allow us to conclude that  $M_{\vec{z}}$ is a stationary subset of $S^{\kappa^+}_\kappa$.  

(2) Since $M_{\vec{z}}\subseteq A_{\vec{z}} \in I[\kappa^+]$, the statement $M_{\vec{z}} \in I[\kappa^+]$ holds trivially. 

 (3) Now, suppose that $S \in I[\kappa^+]$ is a stationary subset of $S^{\kappa^+}_\kappa$. By earlier remarks, there is a sequence $\vec{u} = \seq{u_\alpha}{\alpha< \kappa^+}$ of elements of $[\kappa^+]^{{<}\kappa}$ and club subset $D$ of $\kappa^+$ with the property that every $\gamma \in D \cap S$ is approachable with respect to $\vec{u}$. Define $$E ~ = ~ \Set{\gamma \in S}{\text{Hull}^{(\HH{\theta},\in,\vec{u},\vec{z},D)}(\gamma) \cap \kappa^+ = \gamma}.$$ Then $S\setminus E$ is non-stationary. 
 Fix $\gamma \in E$ and set $M(\gamma)= \text{Hull}^{(\HH{\theta},\in,\vec{u},\vec{z},D)}(\gamma)$. Since $\gamma \in S\in I[\kappa^+]$, there is a cofinal sequence $\vec{\beta}=\seq{\beta_\ell}{\ell < \kappa}$ in $\gamma$ such that every proper initial segment of $\vec{\beta}$ appears in $\vec{u}\restriction\gamma$. 
 But since $\vec{z}$ enumerates all of $[\kappa^+]^{{<}\kappa}$, the fact that $\vec{u}, \vec{z} \in M(\gamma) \prec (\HH{\theta},\in)$ implies that  for every $\alpha < \gamma$ there is a $k(\alpha) < \gamma$ with $u_\alpha = z_{k(\alpha)}$, i.e. $$\Set{u_\alpha}{\alpha < \gamma} ~  \subseteq ~ \Set{z_\alpha}{\alpha < \gamma}. $$ In particular, every proper initial segment of $\vec{\beta}$ appears in $\vec{z}$ before $\gamma$ and therefore $\gamma$ is approachable with respect to $\vec{z}$. 
  These computations show that $S\setminus M_{\vec{z}}\subseteq S\setminus E$ is non-stationary in $\kappa^+$.

 (4) Now, assume that $\kappa^{{<}\kappa} = \kappa$.  Then $\betrag{\eta^{{<}\kappa}} = \kappa$ for every $\eta < \kappa^+$ and hence there is a function $\map{f}{\kappa^+}{\kappa^+}$ with the property that for all $\eta<\kappa^+$, every element of $[\eta]^{{<}\kappa}$ is enumerated by $\vec{z} \restriction f(\eta)$. 
    Let $D$ denote the club of all $\kappa<\gamma < \kappa^+$ such that $$\text{Hull}^{(\HH{\theta},\in,\vec{z},f)}(\gamma) ~ \cap ~ \kappa^+ ~ = ~  \gamma.$$  Pick $\gamma \in D \cap S^{\kappa^+}_\kappa$, and set $W(\gamma)= \text{Hull}^{(\HH{\theta},\in,\vec{z},f)}(\gamma)$. Fix a cofinal sequence $\vec{\alpha}$ in $\gamma$ of order-type $\kappa$  in $\gamma$, and some $\xi < \kappa$. 
     Since $\cof{\gamma} = \kappa$, there is $\eta < \gamma$ with $\alpha_\ell<\eta$ for all $\ell<\xi$ and $\vec{\alpha} \restriction \xi = z_\zeta$ for some $\zeta < f(\eta)$.  Moreover, since $\eta \in W(\gamma)$ and $\betrag{f(\eta)}\leq\kappa\subseteq W(\gamma)$, elementarity implies that $f(\eta) \in W(\gamma)$ and $f(\eta)\subseteq W(\gamma)$. Since $\vec{z}$ and $\zeta$ are both elements of $W(\gamma)$, we can conclude that $z_\zeta \in W(\gamma)$.  Hence $\gamma$ is approachable with respect to $\vec{z}$.
\end{proof}

The next few lemmas address stationary set preservation when GCH may fail to hold.

\begin{lemma}\label{lem_IA_projectiveOverAppIdeal}
The class $\text{IA}_{\kappa}$ is projective stationary over $$\calS ~ = ~\Set{T \subseteq S^{\kappa^+}_\kappa}{\textit{$T$ is stationary and $T \in I[\kappa^+]$}},$$ {i.e.}  if $T \in \calS$, then for every sufficiently large regular cardinal $\theta$ and every function $\map{F}{[\HH{\theta}]^{{<}\omega}}{\HH{\theta}}$, there exists $W \in \text{IA}_\kappa$ such that $W \cap \kappa^+ \in T$ and $W$ is closed under $F$.  
\end{lemma}

\begin{proof}
 Fix $T\in\calS$. Then there is a club $D$ in $\kappa^+$ and a sequence $\vec{z}\in{}^{\kappa^+}[\kappa^+]^{{<}\kappa}$ such that every element of $D \cap T$ is approachable with respect to $\vec{z}$. 
 
Fix a regular $\vartheta$ with $F \in\HH{\vartheta}$, let $\lhd$ be a well-ordering of $\HH{\vartheta}$ and set $$\mathfrak{A} ~ = ~ (\HH{\vartheta},\in,\lhd,\vec{z},D,F, T).$$  Pick $\gamma \in D \cap T$ and $W\prec\mathfrak{A}$ with $\gamma=W \cap \kappa^+$, which is possible because $T$ is stationary.  
 Since $\gamma \in T \cap D$, there is an increasing sequence $\vec{\beta} = \seq{\beta_\alpha}{\alpha < \kappa}$ that is cofinal in $\gamma$ and has the property that every proper initial segment of $\vec{\beta}$ is equal to $z_\alpha$ for some $\alpha < \gamma$.  
 Since $\vec{z} \in W$ and $W \cap \kappa^+ =\gamma$, it follows that every proper initial segment of $\vec{\beta}$ is an element of $W$. 
 Recursively define a sequence $\vec{N} = \seq{N_\alpha}{\alpha < \kappa}$ as follows: 
\begin{itemize} 
 \item Given $\alpha<\kappa$, let $N_{\alpha+1}$ be the $\lhd$-least element of $[\HH{\theta}]^{{<}\kappa}$ such that $N_{\alpha+1}$ is closed under $F$, $\seq{N_\ell}{\ell \leq\alpha} \in N_{\alpha+1}$, $\alpha \subseteq N_{\alpha+1}$, and $\sup(N_{\alpha+1} \cap \kappa^+) \geq \beta_\alpha$. 
 
 \item If $\alpha<\kappa$ is a limit ordinal, then $N_\alpha = \bigcup\Set{N_\ell}{\ell<\alpha}$. 
\end{itemize} 

Set $N= \bigcup\Set{N_\alpha}{\alpha < \kappa}$.  Then $N \in \text{IA}_\kappa$, $N$ is closed under $F$, and
\begin{equation}\label{eq_ge_gamma}
  \sup(N \cap \kappa^+) ~ \geq ~ \sup_{\alpha < \kappa} \beta_\alpha ~ = ~  \gamma.
\end{equation}

 On the other hand, for each $\alpha < \kappa$, the sequence $\seq{N_\ell}{\ell\leq\alpha}$ is definable in $\mathfrak{A}$ from the parameter $\seq{\beta_\ell}{\ell\leq\alpha}$, which is an element of $W$ by the above remarks.  Hence every proper initial segment of $\vec{N}$ is an element of $W$ and, in particular, we know that  $$\sup(N_\alpha \cap \kappa^+) ~ < ~  \gamma ~ = ~  W \cap \kappa^+$$ for all $\alpha < \kappa$. 
  It follows that $\sup(N \cap \kappa^+) \leq \gamma$.  Combined with \eqref{eq_ge_gamma}, this shows that $\sup(N \cap \kappa^+) = \gamma$. 
   Finally, since $\alpha \subseteq N_{\alpha+1}$ for all $\alpha < \kappa$, it follows that $\kappa \subseteq N$ and hence we know that $N \cap \kappa^+$ is transitive.  This allows us to conclude that $$N \cap \kappa^+ ~ = ~  \sup(N \cap \kappa^+) ~ = ~ \gamma,$$ completing the proof of the lemma. 
\end{proof}

The following lemma is one way to salvage stationary set preservation in the non-GCH context.

\begin{lemma}\label{lem_ProperForIApreservesApproach}
 Let $\PPP$ be a $\text{IA}_\kappa$-proper poset and let $T \subseteq S^{\kappa^+}_\kappa$ be stationary with $T \in I[\kappa^+]$.  Then forcing with $\PPP$ preserves the stationarity of $T$. 
\end{lemma}

\begin{proof}
 Set $\tau= \kappa^+$.  Let  $\dot{C}$ be  a $\PPP$-name for a club in $\tau$, let $p \in \mathbb{P}$ and let $\theta$ be a sufficiently large regular cardinal.  Using Lemma \ref{lem_IA_projectiveOverAppIdeal}, we find $\gamma \in T$ and $W \prec (\HH{\theta},\in,p,\dot{C},\PPP)$ with $W \in \text{IA}_\kappa$ and $W \cap \tau=\gamma$. 
 By our assumptions, there is a $(W,\PPP)$-master condition $q$ below $p$ in $\PPP$. 
  Let $G$ be $\PPP$-generic over $\VV$ with $q\in G$. Then $W[G]\cap\tau=W\cap\tau=\gamma$. Moreover, since $\dot{C} \in W$, we now know that $\dot{C}^G\cap W[G]$ is unbounded in $\gamma$ and hence $\gamma\in\dot{C}^G\cap T$. 
  
  These computations show that, in the ground model $\VV$, we have $$q\Vdash_\PPP\anf{\dot{C}\cap\check{T}\neq\emptyset}$$ for densely-many conditions $q$ in $\PPP$.  
\end{proof}


\section{Generalizing a lemma of Jensen}\label{sec_GeneralizeJensenLemma}

The notion of $\text{IA}_\kappa$-properness, defined in Section \ref{section:Prelim}, is a non-$\GCH$ analogue of the notion of \emph{$\kappa$-properness} introduced in {\cite[Definition 3.4]{MR1877015}} . This notion will be important to proving that tails of the iteration described in Section \ref{sec_MainTechResult} do not add cofinal branches to a certain tree, and that argument will closely follow the corresponding arguments of \cite{MR1877015}.

However, $\text{IA}_\kappa$-properness (in the case $\kappa = \omega_1$) is \textbf{not} sufficient for ensuring the preservation of forcing axioms that we need for the proofs of our main results. 
 There are examples of $\text{IA}_{\omega_1}$-proper forcings that destroy, for example, the Proper Forcing Axiom.\footnote{E.g.\ if $2^{\omega_1} = \omega_2$ then there is a natural $\text{IA}_{\omega_1}$-proper poset that forces the \emph{Approachability Property} to hold at $\omega_2$, hence destroys the Proper Forcing Axiom.   This poset is just the natural poset to shoot an $\omega_1$-club through the set $M$ described in Lemma \ref{lem_CardArith_MaximalElement}. }  On the other hand, ${<}\omega_2$-directed closed posets preserve all standard forcing axioms (see \cite{doi:10.1002/malq.200410101} and \cite{MR1782117}). 
  In this section, we generalize a result of Jensen, yielding a property that is forcing equivalent to ${<}\omega_2$-directed closure, but often easier to verify than ${<}\omega_2$-directed closure.

 In \cite{MR2840749}, Jensen defines a poset $\PPP$ to be \textbf{complete} if for every sufficiently large $\theta$, there are club-many $W \in\Poti{\HH{\theta}}{\omega_1}$ such that every $(W,\PPP)$-generic filter has a lower bound in $\PPP$.\footnote{Jensen's notes say this is equivalent to a definition of Shelah in {\cite[Chapter 10]{MR1623206}}.}  He then proves:

\begin{lemma}[Jensen]
The following statements are equivalent for every poset $\PPP$: 
 \begin{enumerate}
  \item The poset $\PPP$ is complete. 
  
  \item The poset $\PPP$ is forcing equivalent to a $\sigma$-closed poset. 
\end{enumerate}
\end{lemma}

We will generalize a version of this lemma to larger cardinals, and, in fact, characterize directed closure (see Lemma \ref{lem_GeneralizeJensenComplete2} below).\footnote{Note that $\sigma$-closure is equivalent to $\sigma$-directed closure, so the distinction is only important at larger cardinals.}  However there are a few technicalities to address. Note that for any $W \in\Poti{\HH{\theta}}{\omega_1}$, the fact that $W$ is countable ensures that there always exist $(W,\PPP)$-generic filters, regardless of what $\PPP$ is.  In particular, the phrase \anf{\emph{\dots every $(W,\PPP)$-generic filter  \dots}} is never vacuous, if $W$ is countable.  
 Of course, for uncountable $W$, it may happen that (depending on the poset $\PPP$) there do not exist any $(W,\PPP)$-generic filters at all; e.g.\ if $W \prec\HH{\theta}$ and $\omega_1 \subseteq W$, then there does not exist a $(W,\Col{\omega}{\omega_1})$-generic filter.

\begin{definition}\label{def_JensenComplete}
 Given a regular uncountable cardinal $\tau$, a poset $\PPP$ is $\boldsymbol{\tau}$\textbf{-Jensen-complete} if the following statements hold for all sufficiently large regular cardinals $\theta$: 
 \begin{enumerate}
  \item\label{item_FilterExists2} For every $p \in \PPP$, there are stationarily-many $W \in\POTI{\HH{\theta}}{\tau}$ with the property that there exists a $(W,\PPP)$-generic filter including $p$. 
  
 \item\label{item_LowerBound} For all but non-stationarily many $W \in\POTI{\HH{\theta}}{\tau}$, every $(W,\PPP)$-generic filter has a lower bound in $\PPP$.\footnote{Note that this clause is allowed to be \textbf{vacuously} true for some elements $W$ of $\POTI{\HH{\theta}}{\tau}$, even for stationarily-many such sets  $W$.}
\end{enumerate}
\end{definition}

\begin{remark}
Note that clause \eqref{item_FilterExists2} of Definition \ref{def_JensenComplete} always holds true for $\tau = \omega_1$, and is hence redundant in that case.  In particular, for $\tau = \omega_1$, Definition \ref{def_JensenComplete} is equivalent to Jensen's definition of completeness.  
\end{remark}

\begin{remark}
 In combination, the clauses \eqref{item_FilterExists2} and \eqref{item_LowerBound} of Definition \ref{def_JensenComplete} imply that the poset $\PPP$ is \emph{totally proper} on a stationary subset of $\POTI{\HH{\theta}}{\tau}$; i.e.\ that there are stationarily-many $W \in \POTI{\HH{\theta}}{\tau}$ such that every condition in $\PPP\cap W$ can be extended to a $(W,\PPP)$-total master condition in the sense of Definition \ref{def_GenericAndTotalMaster}. 
   This conclusion, however, is strictly weaker than $\tau$-Jensen-completeness, since (for example with $\tau = \omega_1$) shooting a club through a bistationary subset of $\omega_1$ has the latter property but is not ${<}\omega_1$-closed. 
    In the case $\tau = \omega_2$, if $2^{\omega_1} = \omega_2$, then shooting an $\omega_1$-club through the set $M$ described in Lemma \ref{lem_CardArith_MaximalElement} is $\text{IA}_{\omega_1}$-totally proper, but forces the approachability property to hold at $\omega_2$.  In particular, this forcing destroys the Proper Forcing Axiom, and $\PFA$ is preserved by $\omega_2$-Jensen-complete forcings (by \cite{doi:10.1002/malq.200410101} and Lemma \ref{lem_GeneralizeJensenComplete2} below). 
\end{remark}

\begin{lemma}\label{lem_JensenCompleteClosureSuccessor}
 If $\kappa$ is an infinite cardinal and $\PPP$ is a ${<}\kappa$-closed poset, then clause \eqref{item_FilterExists2} of Definition \ref{def_JensenComplete} holds for $\tau = \kappa^+$ and $\PPP$.  
\end{lemma}

\begin{proof}
 This follows immediately from Lemmas \ref{lem_IA_is_stat} and \ref{lem_ClosureYieldsGenerics2}.
\end{proof}

Next, we state our generalization of Jensen's lemma. Its Corollary \ref{cor_MainCor} will be used in the proof of Theorem \ref{thm_PropertiesOfHytRautPoset} below.

\begin{lemma}\label{lem_GeneralizeJensenComplete2}
 Given a poset $\mathbb{P}$ and a successor cardinal $\tau$, the following statements are equivalent: 
 \begin{enumerate}
   \item The poset $\PPP$ is forcing equivalent to a ${<}\tau$-directed closed poset. 
   
   \item The poset $\PPP$ is forcing equivalent to a $\tau$-Jensen-complete poset. 
\end{enumerate}
\end{lemma}

\begin{proof}
 First, assume that $\PPP$ is  ${<}\tau$-directed closed. 
   Then, in particular, $\PPP$ is ${<}\tau$-closed, and hence Lemma \ref{lem_JensenCompleteClosureSuccessor} ensures that clause \eqref{item_FilterExists2} of Definition \ref{def_JensenComplete} holds for $\PPP$. 
    But then the directed closure of $\PPP$ ensures that \emph{any} $(W,\PPP)$-generic filter for \emph{any} $W \in \POTI{\HH{\theta}}{\tau}$ has a lower bound in $\PPP$, and hence clause \eqref{item_LowerBound} of Definition \ref{def_JensenComplete} holds for $\PPP$ as well.  This shows that $\PPP$ is $\tau$-Jensen-complete.

 Now, suppose that $\PPP$ is $\tau$-Jensen-complete. Let $\map{F}{[\HH{\theta}]^{{<}\omega}}{\HH{\theta}}$ generate a club witnessing clause \eqref{item_LowerBound} of Definition \ref{def_JensenComplete}, i.e.\ whenever $W \in\POTI{\HH{\theta}}{\tau}$ and $W$ is closed under $F$, then any $(W,\PPP)$-generic filter has a lower bound. 
  We may assume that $F$ also codes a well-ordering $\lhd$ of $\HH{\theta}$, i.e.\ if $W$ is closed under $F$,  then $W \prec (\HH{\theta},\in,\lhd)$) holds.

 Define a poset $\QQQ$, whose conditions are pairs $(M,g)$  satisfying the following statements: 
 \begin{itemize}
   \item $M \in\POTI{\HH{\theta}}{\tau}$. 
   
  \item $M$ is closed under $F$. 
  
  \item $g \subseteq M \cap \PPP$ is an $(M,\PPP)$-generic filter.  
\end{itemize}
 and whose ordering is given by: $$(N,h) \leq_\QQQ(M,g)  ~ \Longleftrightarrow ~  N \supseteq M ~ \wedge ~ h \cap M = g.$$   Note that $\leq_\QQQ$ is transitive and clause \eqref{item_FilterExists2} of Definition \ref{def_JensenComplete} ensures that $\QQQ$ is nonempty.

\begin{claim}\label{clm_DirectedClosed}
 $\mathbb{Q}$ is ${<}\tau$-directed closed.
\end{claim}

\begin{proof}[Proof of Claim \ref{clm_DirectedClosed}]
  Let $\Set{(M_i,g_i)}{i \in I}$ be a directed set of conditions in $\QQQ$ with $\betrag{I}<\tau$. 
  Set $M= \bigcup_{i \in I} M_i$ and $g=\bigcup_{i \in I} g_i$.  We will show that $(M,g)$ is a  condition in $\QQQ$ below all $(M_i,g_i)$. 
  
   The regularity of $\tau$ ensures that $M \in \POTI{\HH{\theta}}{\tau}$, and $M$ is closed under $F$,  because each $M_i$ is closed under $F$ and the collection $\seq{M_i}{i \in I}$ is $\subseteq$-directed. 
   In addition, we have $g\subseteq M\cap\PPP$ and $g$ clearly has the property that $D \cap g \neq \emptyset$ for every dense $D \in M$, because each such $D$ lies in some $M_i$ and $g_i\subseteq g$ is an $(M_i,\PPP)$-generic filter. Finally, the fact that $g$ is a filter follows easily from the fact that the given collection is directed and each $g_i$ is a filter. This shows that $(M,g)$ is a condition in $\QQQ$. 

 Now, fix $i \in I$. Then $M \supseteq M_i$, $g \cap M_i$ is a filter on $M_i \cap \PPP$, and $g \cap M_i\supseteq g_i$.     But since $g_i$ is $(M_i,\PPP)$-generic, we know that $g_i$ is a $\subseteq$-maximal filter on $M_i \cap \PPP$.  In particular, we can conclude that $g \cap M_i = g_i$. This computation shows that $(M,g)\leq_\QQQ(M_i,g_i)$.  
\end{proof}

\begin{claim}\label{clm_ForcingEquiv}
  The poset $\QQQ$ is forcing equivalent to $\PPP$.
\end{claim}

\begin{proof}[Proof of Claim \ref{clm_ForcingEquiv}]
 It is easy to see that the boolean completions of $\tau$-Jensen-complete posets are themselves $\tau$-Jensen-complete. Therefore, we may assume that $\PPP$ is a complete boolean algebra. 
   For each condition $(M,g)$ in $\QQQ$, let $p_{M,g}$ be the $\PPP$-greatest lower bound of $g$. This  conditions exists and is non-zero, because $M$ is closed under $F$, $g$ is $(M,\PPP)$-generic, $\PPP$ is a complete boolean algebra, and because of clause \eqref{item_LowerBound} of Definition \ref{def_JensenComplete}.  In the following, we will show that  the map $$\Map{e}{\QQQ}{\PPP}{(M,g)}{p_{M,g}}$$ is a dense embedding, which will finish the proof of the claim. 
   
   First, we show that  $e$ is order-preserving.  Suppose that $(N,h) \leq_\QQQ (M,g)$. Then $N \supseteq M$ and $g = h \cap M$.  Since $g \subseteq h$ and $p_{N,h}$ is a lower bound of $h$, it follows that $p_{N,h}$ is also a lower bound of $g$.  But $p_{M,g}$ is the greatest lower bound of $g$, and hence $$e(N,h) ~ = ~ p_{N,h} ~ \leq_\PPP ~ p_{M,g} ~ = ~ e(M,g).$$ 
 
  Next, we show that $e$ preserves incompatibility. 
    Suppose $(M_0,g_0)$ and $(M_1,g_1)$ are conditions in $\QQQ$ with the property that there is a condition $p$ in $\PPP$ that extends both $e(M_0,g_0)$ and $e(M_1,g_1)$.  
  By clause \eqref{item_FilterExists2} of Definition \ref{def_JensenComplete}, there we can find $W \in \POTI{\HH{\theta}}{\tau}$ such that $p,g_0, g_1,M_0,M_1\in W$, $W$ is closed under $F$, and there exists a $(W,\PPP)$-generic filter $G$ with $p \in G$. 
  Since $W \cap \tau$ is transitive and $\betrag{M_i} < \tau$ for all $i<2$, it follows that $M_0\cup M_1 \subseteq W$. Furthermore, since $p$ is below both $p_{M_0,g_0}$ and $p_{M_1,g_1}$ and  $M_i \cap \PPP\subseteq W \cap \PPP$ for all $i<2$, the fact that $g_0$ and $g_1$ are maximal filters in $M_0\cap\PPP$ and $M_1\cap\PPP$, respectively, implies that $G \cap M_0 = g_0$ and $G \cap M_1 = g_1$. Hence $(W,G)$ is a condition in $\QQQ$ that lies below both $(M_0,g_0)$ and $(M_1,g_1)$. 

  Finally, we show that the range of $e$ is dense in $\PPP$.   Fix a condition $p$  in $\PPP$.  By clause \eqref{item_FilterExists2} of Definition \ref{def_JensenComplete}, there is a $W \in\POTI{\HH{\theta}}{\tau}$ such that $W$ is closed under $F$, and there exists a $(W,\PPP)$-generic filter $G$ with $p \in G$.  Then $(W,G)$ is a condition in $\QQQ$, and $e(W,g)=p_{W,G}$ is stronger than $p$.  
\end{proof}

 This completes the proof of the lemma. 
\end{proof}

\begin{corollary}
 Given a successor cardinal $\tau$, all $\tau$-Jensen-complete posets are ${<}\tau$-distributive. \qed
\end{corollary}

\begin{remark}
Another common way to verify the ${<}\tau$-distributivity of a given poset $\PPP$ is the following weaker version of $\tau$-Jensen completeness: if for every $p \in \PPP$, there are stationarily-many $W \in\POTI{\HH{\theta}}{\tau}$ such that there is a $(W,\PPP)$-\emph{total} master condition below $p$ (see Definition \ref{def_GenericAndTotalMaster}), then $\PPP$ is ${<}\tau$-distributive.  
Note that this weaker version would not suffice for our purposes, however, because we seem to need ${<}\tau$-directed closure (or a close approximation of it) to prove Theorem \ref{thm_PropertiesOfHytRautPoset}.
\end{remark}

\begin{corollary}\label{cor_MainCor}
  Let $\kappa$ be an infinite regular cardinal and set $\tau = \kappa^+$. 
  If $\PPP$ is ${<}\kappa$-closed poset with the property that  for all but non-stationarily many $W \in\POTI{\HH{\theta}}{\tau}$, every $(W,\PPP)$-generic filter has a lower bound in $\PPP$, then the poset $\PPP$ is forcing equivalent to a ${<}\tau$-directed closed poset.
\end{corollary}

\begin{proof}
 By Lemma \ref{lem_JensenCompleteClosureSuccessor}, the ${<}\kappa$-closure of $\PPP$ ensures that clause \eqref{item_FilterExists2} of Definition \ref{def_JensenComplete} holds.  Since clause \eqref{item_LowerBound} of Definition \ref{def_JensenComplete} holds by assumption, this implies $\mathbb{P}$ is $\tau$-Jensen-complete and  Lemma \ref{lem_GeneralizeJensenComplete2} yields the desired conclusion. 
\end{proof}

In particular, if a poset $\PPP$ satisfies the assumptions of the above corollary for $\kappa=\omega_1$, then forcing with $\PPP$ preserves all standard forcing axioms.


\section{The main technical result}\label{sec_MainTechResult}

  In this section, we will prove the main technical result of our paper.
   It directly extends the main results of \cite{MR1877015} and \cite{MR1222536}. In the next section, we will use it to prove the two theorems stated in the introduction.

 \begin{definition}
  If $\kappa$ be an infinite regular cardinal and let $S\subseteq S^{\kappa^+}_\kappa$. Then we let $T(S)$ denote the tree that consists of all $t\in{}^{{<}\kappa^+}\kappa^+$ such that $\dom{t}$ is a successor ordinal, $\ran{t}\subseteq S$, $t$ is strictly increasing, and $t$ is continuous at all points of cofinality $\kappa$ in its domain and is ordered by end-extension.    
 \end{definition}

 Note that, in the situation of the above definition, the tree $T(S)$ has height $\kappa^+$ and contains a cofinal branch if and only if the set $S$ contains a $\kappa$-club.

 We are now ready to state the aspired result.

\begin{theorem}\label{thm_PropertiesOfHytRautPoset}
Given an infinite regular cardinal $\kappa$, there is a partial order $\PPP$ with the following properties: 
\begin{enumerate}
 \item\label{item_Jensen_complete} $\PPP$ is $\kappa^+$-Jensen complete.\footnote{In particular, Lemma \ref{lem_GeneralizeJensenComplete2} shows that the poset $\PPP$ is forcing equivalent to a ${<}\kappa^+$-directed closed poset.}

 \item\label{item_ChainCondition} $\PPP$ satisfies the $(2^\kappa)^+$-chain condition.

 \item If $G$ is $\PPP$-generic over $\VV$, then, in $V[G]$, there is a subtree $T$ of ${}^{{<}\kappa^+}\kappa^+$ of height $\kappa^+$ without cofinal branches such that the following statements hold: 
 \begin{enumerate}
  \item\label{item_ExistsTreeForApproach} If $S$ is bistationary in $S^{\kappa^+}_\kappa$ and  $S^{\kappa^+}_\kappa \setminus S$ contains a stationary set in $I[\kappa^+]$, then there is an order-preserving function from $T(S)$ to $T$. 

  \item\label{item_MaximalSet} Assume that $\kappa^{{<}\kappa} \leq \kappa^+$ holds in $\VV$. If $M\in\VV$ is a maximum element of $I[\kappa^+]\cap\POT{S^{\kappa^+}_\kappa}$ mod $\NS{}$ in $\VV$,\footnote{Such a subset $M$ exists by Lemma \ref{lem_CardArith_MaximalElement}.} then the following statements hold in $\VV[G]$: 
  \begin{enumerate}
   \item $M$ is a maximum element of $I[\kappa^+] \cap \POT{S^{\kappa^+}_\kappa}$ mod $\NS{}$. 
   
   \item If $S$ is a bistationary in $S^{\kappa^+}_\kappa$ and $M \setminus S$ is stationary, then there is an order-preserving function from $T(S)$ to $T$.  
  \end{enumerate}
  
 \end{enumerate}
\end{enumerate}
\end{theorem}

 Note that the above theorem directly generalizes the main result of \cite{MR1877015}: if $\kappa^{{<}\kappa}=\kappa$ holds, then part \eqref{item_CardArith_AP_fail} of Lemma \ref{lem_CardArith_MaximalElement} shows that $S^{\kappa^+}_\kappa$ is an element of $I[\kappa^+]$, and hence $S^{\kappa^+}_\kappa$ is a maximum element of $I[\kappa^+]\cap\POT{S^{\kappa^+}_\kappa}$ mod $\NS{}$. 
Now, if $G$ is $\PPP$-generic over $\VV$ and $T\in\VV[G]$ is the tree given by the theorem, then there is an order-preserving function from the tree $T(S)$ to $T$ in $\VV[G]$  for every bistationary subset $S$ of $S^{\kappa^+}_\kappa$ in $\VV[G]$. 
 In particular, this shows that $T$ is a \emph{$\kappa$-canary tree} (see {\cite[Definition 3.1]{MR1877015}}) in $\VV[G]$, i.e. if $S$ is a stationary subset of $S^{\kappa^+}_\kappa$ and $\PPP$ is a ${<}\kappa^+$-distributive poset that forces $\kappa^+\setminus S$ to contain a club subset, then forcing with $\PPP$ adds a cofinal branch to $T$.

For the remainder of this section, fix an infinite regular cardinal $\kappa$. Until further notice, we do \textbf{not} make any cardinal arithmetic assumptions. 
 In the following, we closely follow the arguments on pages 1684--1692 of Hyttinen-Rautila in \cite{MR1877015}, which assumed $\GCH$ (in particular, their arguments heavily rely on the assumption $\kappa^{{<}\kappa}=\kappa$). We also follow their notation as closely as possible.

\begin{definition}[\cite{MR1877015}]\label{def_Q_0}
 We let $\QQQ_0$ denote the poset that consists of functions $f$ such that $\dom{f}\subseteq S^{\kappa^+}_\kappa$, $\betrag{\dom{f}} \leq \kappa$, $f(\delta)$ is a function from $\delta$ to $\delta$ for all $\delta \in \dom{f}$, and whenever $\delta < \eta$ are both in the domain of $f$, then $f(\delta) \nsubseteq f(\eta)$,\footnote{Since $\map{f(\delta)}{\delta}{\delta}$ and $\map{f(\eta)}{\eta}{\eta}$, this just means that $f(\delta)$ and $f(\eta)$ disagree at some $\xi < \delta$.} 
 and whose ordering is given by reversed inclusion. 
\end{definition}

\begin{proposition}\label{proposition:Q_0Closure}
 The poset $\QQQ_0$ is ${<}\kappa^+$-directed closed. 
\end{proposition}

\begin{proof}
 This statement follows directly from the fact that the union $f$ of a coherent collection of conditions in $\QQQ_0$ still has the required property that $f(\eta) \nsubseteq f(\beta)$ for all $\eta < \beta$ in the domain of $f$, and, if the union is of size less than $\kappa^+$, then the domain of $f$ has size less than $\kappa^+$ too. 
\end{proof}

\begin{definition}[\cite{MR1877015}]
 If $G_0$ is $\QQQ_0$-generic over $\VV$, then, in $\VV[G_0]$, we define the following subtree of ${}^{{<}\kappa^+} \kappa^+$: $$\calT(G_0) ~ = ~  \Set{h \in {}^{{<}\kappa^+} \kappa^+}{\forall \delta \in S^{\kappa^+}_\kappa ~ h\restriction\delta \neq ({\textstyle \bigcup G_0})(\delta)}.$$
\end{definition}

In the following, we let $\dot{\calT}(\dot{G}_0)$ denote the canonical $\QQQ_0$-name for $\calT(G_0)$.

\begin{remark}
 In the situation of the above definition, the tree $\mathcal{T}(G_0)$ has height $\kappa^+$, since for any $\beta \in S^{\kappa^+}_\kappa$, the function $f$ with domain $\beta$ and constant value $\beta + 1$ has the property that for all $\delta\in S^{\kappa^+}_\kappa$ with $\delta\leq\dom{f}$, the restriction $f\restriction\delta$ is not a function from $\delta$ to $\delta$ and hence cannot be the same as the function $(\bigcup G_0)(\delta)$. 
\end{remark}

\begin{lemma}\label{lem_PropertiesQ_0}
  If $G_0$ is $\QQQ_0$-generic over $\VV$, then the tree  $\calT(G_0)$ has no cofinal branches in $\VV[G_0]$.
\end{lemma}

\begin{proof}
 Work in $\VV$ and assume, towards a contradiction, that a condition $f$ in $\QQQ_0$ forces a $\QQQ_0$-name $\dot{b}$ to be a cofinal branch through $\dot{\calT}(\dot{G}_0)$. Using Proposition \ref{proposition:Q_0Closure},  easy closure arguments allow us to find $\lambda\in S^{\kappa^+}_\kappa$, a function $\map{h}{\lambda}{\lambda}$ and a condition $g$ below $f$ in $\QQQ_0$ such that $\lambda=\sup(\dom{g})$ and $g$ forces $h$ to be the restriction of $\dot{b}$ to $\lambda$. By the definition of $\calT(G_0)$, this implies that $h\restriction\delta\neq g(\delta)$ holds for all $\delta\in\dom{g}$ and we can conclude that $g\cup\{(\lambda,h)\}$ is a condition in $\QQQ_0$ below $g$. But this condition forces that $h$ is not contained in $\dot{\calT}(\dot{G}_0)$, a contradiction. 
\end{proof}

The following poset, again taken from \cite{MR1877015}, adds an order preserving function from $T(S)$ to $\mathcal{T}(G_0)$.  The role of clause \ref{item_OrderPresFcn} is to add such a function with initial segments.  However, the role of clauses  \ref{item_ReqOnX} through \ref{item_ImagesUpperBoundT_G_0} is not obvious; roughly, with the exception of clause \ref{item_Req_S_node}, these properties allow us to verify $\kappa^+$-Jensen completeness by ensuring the existence of lower bounds for any generic filter over any $\kappa$-sized elementary submodel.  The role of clause \ref{item_Req_S_node} is to ensure that no cofinal branch is added to $\mathcal{T}(G_0)$.

\begin{definition}[\cite{MR1877015}]\label{def_PosetAddFunction}
 Let $G_0$ be $\QQQ_0$-generic over $\VV$ and work in an outer model of $\VV[G_0]$\footnote{I.e. a model of $\ZFC$ in which $\VV[G]$ is a transitive class.} with the same bounded subsets of $\kappa^+$ as $\VV$. 
 Let $S$ be a subset of $S^{\kappa^+}_\kappa$. 
 
 \begin{enumerate}
  \item An element $t$ of $\calT(G_0)$ is an $\boldsymbol{S}$\textbf{-node} if $t[\delta]\nsubseteq\delta$ holds for all $\delta \in S^{\kappa^+}_\kappa \setminus S$.  
  
  \item Given a partial function $\pmap{h}{T(S)}{ \calT(G_0)}{part}$, we define $$o(h) ~ = ~ \sup\Set{\dom{t}}{t\in\ran{h}}.$$
  
  \item We let $\PPP(S,G_0)$ denote the unique poset defined by the following clauses: 
  \begin{enumerate}
   \item\label{item_ConditionPSG0} A condition in $\PPP(S,G_0)$ is a pair $(h,X)$ satisfying the following statements: 
   \begin{enumerate}
     \item\label{item_OrderPresFcn} $h$ is an order-preserving partial function of cardinality at most $\kappa$ from the tree $T(S)$ to the tree $\calT(G_0)$ with the property that $\dom{h}$ is closed under initial segments. 

 \item\label{item_ReqOnX} $X$ is a partial function from $\kappa^+$ to $\bigcup\Set{{}^{\beta+1}\kappa^+}{\beta < \kappa^+}$ of cardinality at most $\kappa$ such that $$ o(h) \cap S^{\kappa^+}_\kappa \subseteq \dom{X}$$  and $$X(\alpha) \subseteq ( {\textstyle \bigcup G_0})(\alpha)$$ for all $\alpha \in \dom{X}\cap S^{\kappa^+}_\kappa$. 
 
 \item\label{item_ReqLevelPres} $\dom{h(t)}= \sup(\ran{t})$ for all $t \in \dom{h}$. 

 \item\label{item_Req_S_node} $h(t)$ is an $S$-node for all $t \in \dom{h}$. 

 \item\label{item_range_g_no_X} $X(\alpha) \nsubseteq h(t)$ for all $t\in\dom{h}$ and $\alpha \in\dom{X}$.   

 \item\label{item_ImagesUpperBoundT_G_0} If $\seq{t_\zeta}{\zeta < \kappa}$ is a strictly increasing sequence of elements of $\dom{h}$, then $\bigcup_{\zeta < \kappa} h(t_\zeta) \in \calT(G_0)$.  
  \end{enumerate}

 \item A condition $(h,X)$ is stronger than a condition $(k,Y)$ if and only if $k \subseteq h$ and $Y \subseteq X$ hold. 
  \end{enumerate}

 \item\label{defi_OrderOfCondition}  The \textbf{order} of a condition $p=(h,X)$ in $\PPP(S,G_0)$, denoted by $o(p)$, is defined to be the ordinal $$\max \{ \textstyle{\bigcup \dom{X}}, ~ \textstyle{\bigcup} \Set{\dom{h(t)}}{t\in\dom{h}}\}.$$ 
 \end{enumerate}
\end{definition}

\begin{remark}\label{remark_ReqOnX}
 In the above definition, the requirements on $X(\alpha)$ differ depending on whether or not $\cof{\alpha}=\kappa$. 
 If $\alpha\in S^{\kappa^+}_\kappa \cap \dom{X}$, then $X(\alpha)$ is a proper initial segment of the function $ (\bigcup G_0)(\alpha)$.\footnote{The domain of $X(\alpha)$ is required to be a successor ordinal, so $X(\alpha)$ cannot be the entire function $ (\bigcup G_0)(\alpha)$.} 
  In combination with requirement \eqref{item_range_g_no_X} in the above definition, this shows that for all $\alpha\in\dom{X}\cap S^{\kappa^+}_\kappa$, there is an ordinal $\eta_\alpha<\alpha$ such that no node in the range of $h$ can extend $(\bigcup G_0)(\alpha)\restriction\eta_\alpha$. 
On the other hand, if $\alpha \in\dom{X}$ with $\cof{\alpha}<\kappa$, then the only requirement on $X(\alpha)$ is that nothing in the range of $h$ is allowed to extend $X(\alpha)$. 
\end{remark}

As pointed out near the bottom of page 1684 of  \cite{MR1877015}, the poset $\PPP(S,G_0)$ is ${<}\kappa$-closed. 
Requirements \eqref{item_ReqLevelPres} and \eqref{item_ImagesUpperBoundT_G_0} of Definition \ref{def_PosetAddFunction} are mainly needed for the proof of  Lemma \ref{lem_DensityLemma} below.

\begin{lemma}\label{lem_DensityLemma}
 Let $G_0$ be $\QQQ_0$-generic over $\VV$, let $\VV_1$ be an outer model of $\VV[G_0]$ with the same bounded subsets of $\kappa^+$ as $\VV$, and let  $K$ be $\PPP(S,G_0)^{\VV_1}$-generic over $\VV_1$ for some set $S$ that is bistationary in $S^{\kappa^+}_\kappa$ in $\VV_1$. In $\VV_1[K]$, define $$h_K ~ = ~ \bigcup\Set{h}{(h,X) \in K}$$ and $$X_K ~ = ~ \bigcup\Set{X}{(g,X)\in K}.$$ 
Let $\delta=(\kappa^+)^\VV$.  Then the following statements hold: 
\begin{enumerate}
 \item\label{item_Total} $h_K$ is a total, order-preserving function from $(T(S))^{\VV_1}$ to $(\calT(G_0))^{\VV[G_0]}$ whose range consists entirely of $S$-nodes. 

 \item\label{item_X_K_total} $X_K$ is a total function from $\delta$ to $({}^{{<}\delta}\delta)^\VV$. 

 \item\label{item_NothingRangehK} No element of $\ran{h_K}$ extends an element of $\ran{X_K}$. 

 \item\label{item_DensityIncSeq} Suppose $M$ is an outer model of $\VV_1[K]$ and $\vec{y}=\seq{y_\zeta}{\zeta<\lambda}$ is an increasing sequence of nodes in $\ran{h_K}$ in $M$.  If we set $f_{\vec{y}}=\bigcup_{\zeta < \lambda} y_\zeta$, then $$f_{\vec{y}} \restriction \gamma ~ \neq ~ (\textstyle{\bigcup G_0})(\gamma)$$ for all $\gamma \in (S^\delta_\kappa)^\VV$. 
\end{enumerate}
\end{lemma}

\begin{proof}
 Statement \eqref{item_Total} is {\cite[Claim 3.11]{MR1877015}}.\footnote{This was the only place in the argument where requirement \eqref{item_ReqLevelPres} from Definition \ref{def_PosetAddFunction}  played a role. This requirement was used to fix the error from \cite{MR1222536}. It ensures that the function $h_K$ is total.} 
Statement \eqref{item_X_K_total} is an easy density argument, and statement  \eqref{item_NothingRangehK} follows directly from requirement \ref{item_range_g_no_X} of Definition \ref{def_PosetAddFunction}.
For Statement \eqref{item_DensityIncSeq}, let $M$ and $\vec{y} \in M$ be as stated, and suppose for a contradiction there exists $\gamma<\delta$ with $\cof{\gamma}^\VV=\kappa$ and $f_{\vec{y}} \restriction \gamma = (\textstyle{\bigcup G_0})(\gamma)$.   

Work in $M$.  The statements \eqref{item_Total},  \eqref{item_X_K_total}, and  \eqref{item_NothingRangehK} are obviously upward absolute from $\VV_1[K]$ to $M$.
 By Statement \eqref{item_X_K_total}, we know that $\gamma$ is in the domain of $X_K$, and, 
 by applying Remark \ref{remark_ReqOnX} to some  condition in $K$ whose second coordinate has $\gamma$ in its domain, we can find $\rho_\gamma < \gamma$ with $X_K(\gamma) =  (\textstyle{\bigcup G_0})(\gamma) \restriction \rho_\gamma$. 
Note that, by the definition of $\QQQ_0$, we know that $(\bigcup G_0)(\gamma)$ is a total function on  $\gamma$, and therefore our assumption implies that $\gamma \leq \dom{f_{\vec{y}}}$. 
 Then $\rho_\gamma \in \dom{f_{\vec{y}}}$, and  there is some $\zeta_*<\lambda$ with $\rho_\gamma\in\dom{y_{\zeta_*}}$. 
  In particular, we know that  $$y_{\zeta_*} \restriction \rho_\gamma ~ = ~ f_{\vec{y}} \restriction \rho_\gamma ~ = ~  (\textstyle{\bigcup G_0})(\gamma)\restriction\rho_\gamma ~ = ~ X_K(\gamma).$$
 But this implies that $y_{\zeta_*}\in\ran{h_K}$ extends $X_K(\gamma)\in\ran{X_K}$, contradicting Statement \eqref{item_NothingRangehK}. 
\end{proof}

We now describe the iteration that will witness the poset from Theorem \ref{thm_PropertiesOfHytRautPoset}. This is a slight variant of the iteration described at the bottom of page 1684 of \cite{MR1877015}.  
The main differences are: 
 \begin{itemize}
  \item The length of our iteration is at least $2^{2^\kappa}$.   This is to allow for the case when, in the ground model, the cardinal $2^{\kappa^+}$ is very large. 
  
 \item More significantly, at a given stage $\alpha$ of our iteration, when considering the set $\dot{S}_\alpha$ given to us by the bookkeeping device, we only force with the poset $\PPP(\dot{S}_\alpha,G_0)$ if the statement
 \begin{equation}\label{eq_ComplementApproachable}
  \anf{\textit{$S^{\kappa^+}_\kappa \setminus \dot{S}_\alpha$ contains a stationary set in $I[\kappa^+]$}}
\end{equation}
holds in the corresponding generic extension of the ground model. 
This will ensure (via an application of Lemma \ref{lem_ProperForIApreservesApproach}) that the complements of the $\dot{S}_\alpha$'s remain stationary throughout the iteration, which in turn will be the key to showing that the tree $\calT(G_0)$ has no cofinal branch in the final model. 
\end{itemize}

\begin{remark}
In the $\GCH$ setting of \cite{MR1877015}, requiring \eqref{eq_ComplementApproachable} to hold is no restriction at all, since in that scenario, this statement holds for \textbf{every} set bistationary in $S^{\kappa^+}_\kappa$.  
 But, if $\kappa^{{<}\kappa}  > \kappa$ holds, then the requirement \eqref{eq_ComplementApproachable} seems to be needed in order to prove that the iteration adds no cofinal branch to the tree $\calT(G_0)$. 
\end{remark}

In the following, we fix a cardinal $\varepsilon$ satisfying $\varepsilon^{2^\kappa}=\varepsilon$. 
 Let $\calC$ denote the set of all partial functions from $\varepsilon$ to $\HH{\kappa^+}$ of cardinality at most $\kappa$. 
  Then our assumptions on $\varepsilon$ imply that $\betrag{\calC}=\varepsilon$. 
  Next, let $\calN$ denote the set of all partial functions from ${S^{\kappa^+}_\kappa}\times 2^\kappa$ to $\calC$. Again, our assumptions imply that $\betrag{\calN}=\varepsilon$ and we can pick an $\varepsilon$-to-one surjection $\map{b}{\varepsilon}{\calN}$. 
%

\begin{definition}\label{definition:Iteration}
 We define $$\seq{\PPP_\alpha, \dot{\QQQ}_\xi}{\alpha \leq \varepsilon, ~ \xi < \varepsilon}$$ to be a  ${<}\kappa^+$-support iteration satisfying the following clauses: 
\begin{enumerate}
 \item $\dot{\QQQ}_0$ is chosen in a canonical way that ensures that the map $$\Map{i}{\QQQ_0}{\PPP_1}{q}{\langle\check{q}\rangle}$$ is an isomorphism. 

 \item\label{item-Bookkeeping} Assume that $\alpha \in [1, \varepsilon)$ has the property that the poset $\PPP_\alpha$ is ${<}\kappa^+$-distributive and there exists a sequence $\seq{q_{\gamma,\xi}}{(\gamma,\xi)\in\dom{b(\alpha)}}$ of conditions in $\PPP_\alpha$ such that the following statements hold: 
   \begin{itemize}
    \item If $(\gamma,\xi)\in\dom{b(\alpha)}$, then $\supp{q_{\gamma,\xi}}=\dom{b(\alpha)(\gamma,\xi)}$. 
    
    \item If $(\gamma,\xi)\in\dom{b(\alpha)}$, $\ell\in\supp{q_{\gamma,\xi}}$ and $b(\alpha)(\gamma,\xi)(\ell)=x$, then $q_{\gamma,\xi}(\ell)=\check{x}$. 
   \end{itemize}
   
 If we define $$\dot{B}_\alpha ~ = ~ \Set{(\check{\gamma},q_{\gamma,\xi})}{(\gamma,\xi)\in\dom{b(\alpha)}},$$ then there exists a $\PPP_\alpha$-name $\dot{S}_\alpha$ for a subset of $S^{\kappa^+}_\kappa$ such that the following statements hold in $\VV[G]$ whenever $G$ is $\PPP_\alpha$-generic over $\VV$ and $G_0$ is the induced $\QQQ_0$-generic filter over $\VV$: 
 \begin{enumerate}
   \item $\dot{\QQQ}_\alpha^G=\PPP(\dot{S}_\alpha^G,G_0)^{\VV[G]}$. 
 
   \item If the set $S^{\kappa^+}_\kappa \setminus \dot{B}_\alpha^G$ contains a stationary set in $I[\kappa^+]$, then $\dot{S}_\alpha^G=\dot{B}_\alpha^G$. 
   
   \item\label{item-def-iteriation-trivial-case-1} If the set $S^{\kappa^+}_\kappa \setminus \dot{B}_\alpha^G$ does not contain a stationary set in $I[\kappa^+]$, then $\dot{S}_\alpha^G=\emptyset$.  
 \end{enumerate}
 
  \item\label{item-def-iteriation-trivial-case-2} Assume that $\alpha \in [1, \varepsilon)$ has the property that the poset $\PPP_\alpha$ is ${<}\kappa^+$-distributive and there exists no sequence of conditions in $\PPP_\alpha$ with the properties listed in \eqref{item-Bookkeeping}. Then $\dot{S}_\alpha=\emptyset$ and $\dot{\QQQ}_\alpha^G=\PPP(\dot{S}_\alpha^G,G_0)^{\VV[G]}$ whenever $G$ is $\PPP_\alpha$-generic over $\VV$ and $G_0$ is the induced $\QQQ_0$-generic filter over $\VV$.  
 
  \item If $\alpha \in [1, \varepsilon)$ has the property that the poset $\PPP_\alpha$ is not ${<}\kappa^+$-distributive, then $\dot{\QQQ}_\alpha$ is a $\PPP_\alpha$-name for a trivial poset.  
\end{enumerate} 
 \end{definition}

\begin{remark}
 We include the cases \eqref{item-def-iteriation-trivial-case-1} and \eqref{item-def-iteriation-trivial-case-2} in the above definition of the name $\dot{S}_\alpha$ to simplify notation later on.  Note that, since we have $T(\emptyset) = \emptyset$, conditions in $\PPP(\emptyset, G_0)$ always have trivial first coordinate, and the poset $\PPP(\emptyset, G_0)$ is forcing equivalent to $\Add{\kappa^+}{1}$. 
\end{remark}

Throughout the rest of this paper, $\PPP$ refers to the poset  $\PPP_\varepsilon$. Moreover, in order to conform to the notation from \cite{MR1877015}, if $G$ is $\PPP_\alpha$-generic over $\VV$ for some $\alpha\leq\varepsilon$, then we let $G_0$ denote the induced $\QQQ_0$-generic filter over $\VV$.

\begin{definition}\label{def_Flat}
 A condition $p$ in $\PPP$ is called \textbf{flat} if there exists a sequence $\seq{x_\alpha}{i\in\supp{p}}$ with the property that $x_\alpha\in\HH{\kappa^+}$ and $$p\restriction\alpha \Vdash_{\PPP_\alpha} \anf{p(\alpha)=\check{x}_\alpha}$$ hold for all $\alpha\in\supp{p}$ with the property that $\PPP_\alpha$ is ${<}\kappa^+$-distributive.  
\end{definition}

Just as in \cite{MR1877015}, the flat conditions turn out to be dense in $\PPP$, as we will see in Lemma \ref{lem_P_Jensen_complete} below. 
 Although the density of the flat conditions is not needed to prove the $\kappa^+$-Jensen-completeness of $\PPP$ in Lemma \ref{lem_P_Jensen_complete}, it will be crucial for the proofs of the following statements: 
\begin{itemize}
 \item The \anf{tails} of the above iteration are proper with respect to $\text{IA}_\kappa$ (see Lemma \ref{lem_TailsProperIA}), which in turn is important for the proof that the tree $\calT(G_0)$ has no cofinal branches in $\PPP$-generic extensions of $\VV$ (see Lemma \ref{lem_NoCofBranch}). 

 \item If $2^\kappa = \kappa^+$, then $\PPP$ satisfies the $\kappa^{++}$-chain condition. 
\end{itemize}

The function $p(f_0,g)$ defined in Definition \ref{def_ComponentNotation_NoTails} below is a natural attempt to form a flat condition out of a $(W,\PPP)$-generic filter for some elementary substructure $W$ of size $\kappa$.

\begin{definition}\label{def_ComponentNotation_NoTails}
 Suppose  $W \prec (\HH{\theta},\in,\PPP)$ with $\betrag{W}=\kappa \subseteq W$, and $g \subseteq\PPP\cap W$ is a $(W,\PPP)$-generic filter in $\VV$. 
 \begin{enumerate}
   \item Set $g_0=\Set{q\in\QQQ_0}{\exists p\in g ~ p(0)=\check{q}}$.\footnote{This notation is chosen to keep in line with the notational convention from \cite{MR1877015} of identifying $\PPP_1$ with $\QQQ_0$ and  referring to the induced $\QQQ_0$-generic filter by $G_0$.  Notice that $\bigcup g_0$ is easily a condition in $\QQQ_0$.} 
    
  \item Given $0<\alpha \leq  \varepsilon$,  we define
$$g_\alpha ~ = ~ \Set{ p \restriction \alpha}{p \in g}.$$ 

 \item Given $0<\alpha <  \varepsilon$ with the property that the poset $\PPP_\alpha$ is ${<}\kappa^+$-distributive, let $\dot{c}_{g,\alpha}$, $\dot{h}_{g,\alpha}$ and $\dot{X}_{g,\alpha}$ denote the canonical $\PPP_\alpha$-names with the property that,\footnote{The idea behind this definition is that $\dot{c}_{g,\alpha}$ names the evaluation of the \emph{$\alpha$-th component} of $g$ after forcing with $\PPP_\alpha$ over $\VV$, and $\dot{h}_{g,\alpha}$ and $\dot{X}_{g,\alpha}$ name the unions of the left and right components (respectively) of that $\alpha$-th component of $g$.} 
 whenever $G$ is $\PPP_\alpha$-generic over $\VV$, then 
  \begin{itemize}
   \item $\dot{c}_{g,\alpha}^G ~ = ~ \Set{p(\alpha)^G}{p \in g}$, 

   \item $\dot{h}_{g,\alpha}^G ~ = ~  \bigcup\Set{h}{\exists X ~ (h,X)\in\dot{c}_{g,\alpha}^G}$, and  

   \item $\dot{X}_{g,\alpha}^G ~ = ~ \bigcup\Set{X}{\exists h ~ (h,X)\in\dot{c}_{g,\alpha}^G}$. 
  \end{itemize}

 \item If $f_0$ is a condition in $\QQQ_0$ extending $\bigcup g_0$, then we define a function $p(f_0, g)$ with domain $W \cap \varepsilon$ by setting  $$p(f_0, g) ~ = ~  \langle f_0 \rangle^\frown \seq{\mathsf{pair}_{\PPP_\alpha}(\dot{h}_{g,\alpha},\dot{X}_{g,\alpha})}{\textit{$\alpha \in W \cap [1,\varepsilon)$ with $\PPP_\alpha$ ${<}\kappa^+$-distributive}},$$ where $\mathsf{pair}_{\PPP_\alpha}(\dot{h}_{g,\alpha},\dot{X}_{g,\alpha})$ denotes the canonical $\PPP_\alpha$-name for the ordered pair of $\dot{h}_{g,\alpha}$ and $\dot{X}_{g,\alpha}$. 
\end{enumerate} 
\end{definition}

Note that, in the last part of the above definition, the function $p(f_0, g)$ may or may not be a condition in $\PPP$. The following lemma shows how we can ensure that  $p(f_0, g)$ is a flat condition below every condition in $g$.

\begin{lemma}\label{lem_KeyBackgroundLemma_NoTails}
Suppose $W$, $g$, and $f_0$ are as in Definition \ref{def_ComponentNotation_NoTails}.  Then one of the following statements holds:  
\begin{enumerate}
 \item\label{item_GoodCase} $p(f_0, g)$ is a flat condition in $\PPP$ that extends every element of $g$.

 \item\label{item_BadCase} There is an $\alpha\in W \cap [1,\varepsilon)$ such that the following statements hold: 
 \begin{enumerate}
   \item\label{item_Condition_extending_g}   $p(f_0, g)\restriction\alpha$ is a condition in $\PPP_\alpha$ that extends every element of $g_\alpha$.

   \item\label{item_NoNewKappaSeq} If $\tau \in W$ is a $\PPP_\alpha$-name for a function from $\kappa$ to the ordinals, then  $$p(f_0,g)\restriction\alpha\Vdash_{\PPP_\alpha}\anf{\tau\in\check{W}}.$$ 

   \item\label{item_WhatStrongerConditionForces} There is a condition $q$ in $\PPP_\alpha$ below $p(f_0,g)\restriction\alpha$ with the property that the following statements hold true in $\VV[G]$, whenever $G$ is $\PPP_\alpha$-generic over $\VV$ with $q\in G$:   
   \begin{enumerate}
  \item\label{item_CofIsKappa} $\cof{W\cap\kappa^+}^\VV=\kappa$. In particular, we have $W\cap\kappa^+\in\dom{\bigcup G_0}$. 
  
  \item\label{item_EveryProperIS} Every proper initial segment of $(\bigcup G_0)(W\cap\kappa^+)$ is an element of $W$, and is an $\dot{S}_\alpha^G$-node. 
  
  \item\label{item_NoProperInitial} No proper initial segment of $(\bigcup G_0)(W\cap\kappa^+)$ is an element of $\ran{\dot{X}_{g,\alpha}^G}$. 
   \end{enumerate}
 \end{enumerate}
 \end{enumerate} 
\end{lemma}

\begin{proof}
 Set $g_\varepsilon=g$ and, given $\beta\leq\varepsilon$, let $\Phi_\beta$ denote the statement asserting that $p(f_0,g)\restriction\beta$ is a flat condition in $\PPP_\beta$ that lies below every element of $g_\beta$.   
 Suppose that $\Phi_\varepsilon$ fails, {i.e.} that part \eqref{item_GoodCase} of the disjunctive conclusion of the statement of the lemma fails.  Let $\beta \leq \varepsilon$ be the least ordinal such that $\Phi_\beta$ fails. 

\begin{claim}\label{clm_PropOfj}
 $\beta$ is a successor ordinal and an element of $W$. 
\end{claim}

\begin{proof}[Proof of Claim \ref{clm_PropOfj}]
  First, we have $\beta>0$, because the 0th component of $p(f_0,g)$ is $f_0$, which is assumed to be a condition stronger than $\bigcup g_0$, and $g_0$ is a $(W,\QQQ_0)$-generic filter. 
 Now, assume, towards a contradiction, that $\beta$ is a limit ordinal.
 Since $\Phi_\alpha$ holds for all $\alpha<\beta$ and since the support of $p(f_0,g)$ is contained in the $\kappa$-sized set $W\cap\varepsilon$, it follows easily that $p(f_0,g)\restriction\beta$ is a condition and is below every element of $g_\beta$. 
 Furthermore, for each $\alpha \in W \cap\beta$, let $$\vec{x}_\alpha ~ = ~ \seq{x^\alpha_\xi}{\xi \in W \cap \alpha}$$ witness flatness of $p(f_0,g) \restriction\alpha$. Then for all $\alpha_0 <\alpha_1 <\beta$, it follows easily that $x^{\alpha_0}_\xi = x^{\alpha_1}_\xi$ holds for all $\xi \in W \cap \alpha_0$.  So the $\vec{x}_\alpha$'s are coherent, and their union witnesses flatness of $p(f_0,g) \restriction\beta$.  This shows that $\Phi_\beta$ holds, a contradiction. 

 The above computations yield an ordinal $\alpha$ with $\beta=\alpha+1$.    Assume, towards a contradiction, that $\alpha \notin W$.  
 Note that, if $r \in g$, then $r \in W$ and, since $\supp{r}$ is a $\kappa$-sized element of $W$ and $\kappa \subseteq W$, it follows that $\supp{r} \subseteq W$. 
  Since $\alpha$ is not an element of $W$, this shows that $g_\alpha=g_\beta$ and $p(f_0,g)\restriction\beta=p\big(f_0,g)\restriction\alpha$. But, since $\Phi_\alpha$ holds, this immediately implies that $\Phi_\beta$ holds too, a contradiction. 
\end{proof}

The above claim shows that there is an $\alpha\in W \cap [1,\varepsilon)$ with $\beta  = \alpha + 1$.   We claim that this $\alpha$ witnesses part \eqref{item_BadCase} of the conclusion of the lemma holds true. 
By the minimality of $\beta$, we know that \eqref{item_Condition_extending_g} holds and $\PPP_\alpha$ is ${<}\kappa^+$-distributive.  Moreover, since $p(f_0,g) \restriction \alpha$ is a condition that extends the $(W,\PPP_\alpha)$-generic filter $g_\alpha$, part \eqref{item_NoNewKappaSeq} holds by Lemma \ref{lem_TotalMasterDist}.

\begin{claim}\label{clm_ClaimRE_WhatPForces}
There is an $x \in\HH{\kappa^+}$ with 
 \begin{equation}\label{eq_Ground_Model_Element}
  p(f_0,g) \restriction \alpha \Vdash_{\PPP_\alpha}\anf{p(f_0,g)(\alpha)=\check{x}}. 
 \end{equation}
  Furthermore, if $G$ is $\PPP_\alpha$-generic over $\VV$ with $p(f_0,g)\restriction\alpha\in G$, then the following statements hold in $\VV[G]$:  
\begin{enumerate}

 \item\label{item_ExceptPossibly} The pair $(\dot{h}_{g,\alpha}^G,\dot{X}_{g,\alpha}^G)$ satisfies all requirements to be a condition in the poset $\PPP(\dot{S}_\alpha^G,G_0)$, \textbf{with the possible exception of} requirement \eqref{item_ImagesUpperBoundT_G_0} of Definition \ref{def_PosetAddFunction}.  In particular, the following statements hold: 
  \begin{enumerate}
    \item Every element of $\ran{\dot{h}_{g,\alpha}^G}$ is an $\dot{S}_\alpha^G$-node (i.e.\ requirement \eqref{item_Req_S_node} of Definition \ref{def_PosetAddFunction} is satisfied). 
    
  \item No element of $\ran{\dot{h}_{g,\alpha}^G}$ extends an element of $\ran{\dot{X}_{g,\alpha}^G}$ (i.e.\ requirement \eqref{item_range_g_no_X} of Definition \ref{def_PosetAddFunction} is satisfied).
  \end{enumerate}

  \item\label{item_IncSeqRangeDoth} \textbf{If} the pair $(\dot{h}_{g,\alpha}^G,\dot{X}_{g,\alpha}^G)$ is \textbf{not} a condition in $\PPP(\dot{S}_\alpha^G,G_0)$, then the following statements hold:     
  \begin{enumerate}
  \item\label{itemclm_CofIsKappa} $\cof{W \cap \kappa^+}^\VV = \kappa$. In particular, we have $W \cap \kappa^+\in\dom{\bigcup G_0}$). 
  
  \item\label{itemclm_EveryProperIS} Every proper initial segment of $(\bigcup G_0)(W \cap \kappa^+)$ is an element of $W$, and is an $\dot{S}_\alpha^G$-node. 
  
  \item\label{itemclm_NoProperInitial} No proper initial segment of $(\bigcup G_0)(W \cap \kappa^+)$ is an element of $\ran{\dot{X}_{g,\alpha}^G}$. 
   \end{enumerate}
\end{enumerate}
\end{claim}

\begin{proof}[Proof of Claim \ref{clm_ClaimRE_WhatPForces}]
  In order to make use of Lemma \ref{lem_DensityLemma}, it will be more convenient to work with the transitive collapse of $W$ instead of $W$ itself.  Let $H_W$ be the transitive collapse of $W$, and let $\map{\sigma}{H_W}{W \prec\HH{\theta}}$ be the inverse of the collapsing map.  
  In the following, if $b$ is a set, then we will write 
  \begin{equation*}
   \bar{b} ~ = ~  \sigma^{{-}1} [b] ~ \subseteq ~ H_W.
  \end{equation*}
  Note that $\bar{b}=\sigma^{{-}1}(b)$ holds for all $b\in W$ and we will frequently use this abbreviation in the following arguments.  
  
Since $g$ is a $(W,\PPP)$-generic filter, we know that $\bar{g}$ is a $\bar{\PPP}$-generic over $H_W$, and, in particular, it follows that $\bar{g}_\gamma$ is $\bar{\PPP}_\gamma$-generic over $H_W$ for  all $\gamma\in W\cap\varepsilon$.  
 If we define 
  $$k ~ = ~ \Set{q(\bar{\alpha})^{\bar{g}_\alpha}}{q\in\bar{g}_{\alpha+1}} ~ \subseteq ~  H_W[\bar{g}_\alpha],$$ then  $k$ is $\sigma^{{-}1}(\dot{\QQQ}_\alpha)^{\bar{g}_\alpha}$-generic over $H_W[\bar{g}_\alpha]$ with $H_W[\bar{g}_{\alpha+1}] = H_W[\bar{g}_\alpha][k]$.

Set $\delta= (\kappa^+)^{H_W}$, and note that $\delta = \crit{\sigma}$, because $\betrag{W}=\kappa \subseteq W$. 
 Since $\Phi_\alpha$ holds, we know that $p(f_0,g) \restriction \alpha$ is a condition in $\PPP_\alpha$ that extends every element of the $(W,\PPP_\alpha)$-generic filter $g_\alpha$.  In particular, $p(f_0,g) \restriction \alpha$ it is a total master condition for $W$. By Lemma \ref{lem_TotalMasterDist}, every $\PPP_\alpha$-name for a function from $\kappa$ to the ordinals  in $W$  is forced by $p(f_0,g) \restriction \alpha$ to be evaluated to an element of $W$.  It follows that
\begin{equation}\label{eq_DistExt}
 H_W   \cap  {}^\kappa\Ord ~ = ~ H_W[\bar{g}_\alpha]  \cap  {}^\kappa\Ord. 
\end{equation}

Let  $h_k$ be the union of the left coordinates  of $k$ and let $X_k$ be the union of the right coordinates of $k$. 
  By \eqref{eq_DistExt}, we can apply Lemma \ref{lem_DensityLemma} with the ground model $H_W$ and the outer model $H_W[\bar{g}_\alpha]$ and derive the following statements:  
\begin{itemize}
 \item The function $$\map{h_k}{T(\sigma^{{-}1}(\dot{S}_\alpha)^{\bar{g}_\alpha})^{H_W[\bar{g}_\alpha]}}{\sigma^{{-}1}(\dot{\calT}(\dot{G}_0))^{\bar{g}_0}}$$ is 
  order preserving and every element of its range is a $\sigma^{{-}1}(\dot{S}_\alpha)^{\bar{g}_\alpha}$-node in $H_W[\bar{g}_\alpha]$. 

 \item No element of $\ran{h_k}$ extends an element of the range of the function $$\map{X_k}{\delta}{({}^{{<}\delta}\delta)^{H_W}}.$$ 
\end{itemize}

Set $x=(h_k,X_k)$. In the following, we will show that $p(f_0,g) \restriction \alpha$ and $x$ satisfy \eqref{eq_Ground_Model_Element}, and $p(f_0,g) \restriction \alpha$ forces the other statements of the claim to hold true.

Let  $G$ be $\PPP_\alpha$-generic over $\VV$ with $p(f_0,g) \restriction \alpha \in G$.  Work in $V[G]$. 
  Since $\alpha \in W$ and $p(f_0,g) \restriction \alpha$ is a $W$-total master condition that, in particular,  extends every element of $g_\alpha$, it follows that $W[G] \cap \VV = W$, $G \cap W = g_\alpha$, and $\sigma$ can be canonically lifted to an elementary embedding $$\map{\hat{\sigma}}{H_W[\bar{g}_\alpha]}{W[G] \prec \HH{\theta}[G]}.$$ satisfying $$\hat{\sigma}(\tau^{\bar{g}_\alpha}) ~ = ~ \sigma(\tau)^G$$ for every $\bar{\PPP}_\alpha$-name $\tau$ in $H_W$. 
   Note that $\ran{\hat{\sigma}} = W[G]$ holds.  

 Now, pick an element $q$ of $\bar{g}_{\alpha+1}\subseteq H_W$. 
 By the definition of $\hat{\sigma}$, we then have $$\hat{\sigma}(q(\bar{\alpha})^{\bar{g}_\alpha}) ~ = ~ \sigma(q(\bar{\alpha}))^G ~ = ~ (\sigma(q)(\sigma(\bar{\alpha})))^G ~ = ~ (\sigma(q)(\alpha))^G.$$ Since $q\in\bar{g}_{\alpha+1}$ implies that $\sigma(q)\in g_{\alpha+1}$, we can now conclude that $\hat{\sigma}(q(\bar{\alpha})^{\bar{g}_\alpha})\in \dot{c}_{g,\alpha}^G$. These computations show that $\hat{\sigma}[k] \subseteq \dot{c}_{g,\alpha}^G$.

 Next, fix a condition $p$ in $g$. Since $g\cup\{\alpha\}\subseteq W=\ran{\sigma}$, we then have $\bar{p}\restriction(\bar{\alpha}+1)\in\bar{g}_{\alpha+1}$ and $\bar{p}(\bar{\alpha})^{\bar{g}_\alpha}\in k$. 
  By the definition of $\hat{\sigma}$, we now know that $$p(\alpha)^G ~ = ~ \sigma(\bar{p}(\bar{\alpha}))^G ~ = ~ \hat{\sigma}(\bar{p}(\bar{\alpha})^{\bar{g}_\alpha}) ~ \in ~ \hat{\sigma}[k].$$ This shows that $\dot{c}_{g,\alpha}^G\subseteq\hat{\sigma}[k]$ and, together with the above computations, we can conclude that 
   \begin{equation}\label{eq_ImageBarK}
     \hat{\sigma}[k] ~ = ~ \dot{c}_{g,\alpha}^G. 
   \end{equation}

By \eqref{eq_ImageBarK} and the elementarity of $\hat{\sigma}$, we also have $\hat{\sigma}[h_k]=\dot{h}_{g,\alpha}^G$ and $\hat{\sigma}[X_k]=\dot{X}_{g,\alpha}^G$.  
Note that conditions in $\sigma^{{-}1}(\dot{\QQQ}_\alpha)^{\bar{g}_\alpha}$ are  elements of $\HH{\delta}^{H_W[\bar{g}_\alpha]}$ and hence $k$ is a subset of $\HH{\delta}^{H_W[\bar{g}_\alpha]}$. 
 Since the critical point of $\hat{\sigma}$ is $\delta$, it follows that $k$, $h_k$ and $X_k$ are all pointwise fixed by $\hat{\sigma}$. In particular, we have  
\begin{equation}\label{eq_ConditionPair}
    x ~ = ~ (h_k,X_k) ~ = ~ (\dot{h}_{g,\alpha}^G,\dot{X}_{g,\alpha}^G). 
\end{equation}
Since $p(f_0,g)(\alpha)$ is, by definition, the $\PPP_\alpha$-name $\mathsf{pair}_{\PPP_\alpha}(\dot{h}_{g,\alpha},\dot{X}_{g,\alpha})$, this completes the proof of \eqref{eq_Ground_Model_Element}.  

%
%
%
%
%

Part \eqref{item_ExceptPossibly} of the claim follows by the properties of $(h_k,X_k)$ over $H_W[\bar{g}_\alpha]$ discussed above, together with the equality \eqref{eq_ConditionPair}, elementarity of $\hat{\sigma}$, and the fact that $\hat{\sigma}$ fixes bounded subsets of $\delta$ that lie in $H_W[\bar{g}_\alpha]$. 
 For example, to verify requirement \eqref{item_Req_S_node} of Definition \ref{def_PosetAddFunction}, suppose $t$ is in the range of $h_k$.  Then $t$ is in the range of the left coordinate of some condition in $k \subseteq \hat{\sigma}^{{-}1}(\PPP(\dot{S}_\alpha^G,G_0)^{\VV[G]})$, and hence $t$ is an $\hat{\sigma}^{{-}1}(\dot{S}_\alpha^G)$-node in $H_W[\bar{g}_\alpha]$.  By elementarity of $\hat{\sigma}$ and the fact that $\hat{\sigma}$ fixes $y$, it follows that $t$ is an $\dot{S}_\alpha^G$-node in $\VV[G]$. 
  The remaining requirements of Definition \ref{def_PosetAddFunction}, except for requirement \eqref{item_ImagesUpperBoundT_G_0}, are easily verified for the pair displayed in \eqref{eq_ConditionPair} in a similar manner.

Now, we prove that $p(f_0,g) \restriction \alpha$ forces the statements in part \eqref{item_IncSeqRangeDoth} of the claim.  Recall $G$ is an arbitrary $\PPP_\alpha$-generic filter over $\VV$ with $p(f_0,g) \restriction\alpha \in G$. 
 Assume that the ordered pair \eqref{eq_ConditionPair} is \textbf{not} a condition in $\PPP(\dot{S}_\alpha^G,G_0)$ in $\VV[G_i]$. In the following, we show that the statements \eqref{itemclm_CofIsKappa}, \eqref{itemclm_EveryProperIS}, and \eqref{itemclm_NoProperInitial} of the claim hold true in $\VV[G]$. 
  By part \eqref{item_ExceptPossibly} of the claim, it must be requirement \eqref{item_ImagesUpperBoundT_G_0} of Definition \ref{def_PosetAddFunction} that fails. 
  In particular, there is an increasing sequence $\vec{y}=\seq{y_\zeta}{\zeta < \kappa}$ of nodes in the range of $h_k$ in $\VV[G]$ such that the function $f_{\bar{y}}=\bigcup_{\zeta < \kappa} y_\zeta$ is not an element of $\calT(G_0)$. 
   Since the elements of the range of $h_k$ are $\kappa$-sized objects in $H_W[\bar{g}_\alpha]$ and hence in $H_W$ by \eqref{eq_DistExt}, this implies that the domain of $f_{\vec{y}}$ is at most $\delta = (\kappa^+)^{H_W}$.  In summary, there is some ordinal $\gamma\in(S^{\kappa^+}_\kappa)^\VV$ such that 
\begin{equation}
V[G] \models\anf{\textit{$\gamma ~ \leq ~ \delta ~ = ~ (\kappa^+)^{H_W} ~ $ and $ ~ f_{\vec{y}}\restriction \gamma ~ = ~  (\textstyle{\bigcup G_0})(\gamma)$}}.
\end{equation}

By part \eqref{item_DensityIncSeq} of Lemma \ref{lem_DensityLemma} -- viewing $H_W$ as the ground model $\VV$, $H_W[\bar{g}_\alpha]$ as the outer model $\VV_1$, $k$ as the generic filter $K$, and $V[G]$ as the outer model $M$ from the statement of that lemma -- the ordinal $\gamma$ cannot be strictly smaller than $\delta$, and hence we can conclude that $\gamma = \delta$. 
 But this implies that $\cof{W\cap\kappa^+}^\VV=\cof{\delta}^\VV=\cof{\gamma}^\VV=\kappa$, proving part \eqref{itemclm_CofIsKappa} of the claim. 

  In summary, we have shown that $W \cap \kappa^+ = \dom{f_{\vec{y}}}$ and
\begin{equation}\label{eq_f_at_delta_G_0}
   f_{\vec{y}} ~ = ~ (\textstyle{\bigcup G_0})(W \cap \kappa^+). 
\end{equation}
Since $f_{\vec{y}}$ is a union of functions in the range of $h_k$, and \eqref{eq_DistExt} together with the fact that the critical point of $\sigma$ is $\delta$ imply that $h_k \subseteq W$,  
 every proper initial segment of $f_{\vec{y}}$ is an element of $W$. Furthermore, every proper initial segment of $f_{\vec{y}}$ is extended by some $y \in \ran{h_k}$, which is an $\dot{S}_\alpha^G$-node in $\VV[G]$ by part \eqref{item_ExceptPossibly} of the claim.  
 Hence, every proper initial segment of $f_{\vec{y}}$ is an $\dot{S}_\alpha^G$-node in $\VV[G]$, and is an element of $W$.  Together with \eqref{eq_f_at_delta_G_0}, this proves part \eqref{itemclm_EveryProperIS} of the claim.

Finally, to prove part \eqref{itemclm_NoProperInitial} of the claim, suppose $\gamma\in\dom{X_k}$, define $\eta= \dom{X_k(\gamma)}$ and assume, towards a contradiction, that $f_{\vec{y}} \restriction \eta = \bar{X}(\alpha)$. Note that $\eta < \delta$, because $X_k \subseteq H_W[\bar{g}_\alpha]$. In particular, we have $y_\zeta \restriction \eta = X_k(\alpha)$ for some $\zeta<\kappa$. 
 But this contradicts the fact from part \eqref{item_ExceptPossibly} that nothing in the range of $h_k$  extends any function from the range of $X_k$.  
\end{proof}

It remains to prove part \eqref{item_WhatStrongerConditionForces} of the lemma, which will essentially follow from part \eqref{item_IncSeqRangeDoth} of Claim \ref{clm_ClaimRE_WhatPForces}, though we first must dispense with a technicality. 
 Recall that $\Phi_{\alpha+1}$ fails, but $\Phi_\alpha$ holds.  Next, we observe that the failure of $\Phi_{\alpha+1}$ is due to the function $p(f_0,g) \restriction (\alpha+1)$ not being a condition at all (rather than being a condition but failing to extend $g_{\alpha+1}$, or being a condition but failing to be flat):

\begin{claim}\label{clm_SomeExtensionForces}
 Some condition in $\PPP_\alpha$ below $p(f_0,g) \restriction \alpha$ forces that 
  \begin{equation}\label{eq_PairNamed}
    p(f_0,g)(\alpha) ~ = ~ \mathsf{pair}_{\PPP_\alpha}(\dot{h}_{g,\alpha},\dot{X}_{g,\alpha})
  \end{equation}
  is \textbf{not} a condition in $\dot{\QQQ}_\alpha$. 
\end{claim}

\begin{proof}[Proof of Claim \ref{clm_SomeExtensionForces}]
  Assume not, i.e.\ suppose that $p(f_0,g) \restriction \alpha$ forces that the pair in \eqref{eq_PairNamed} to be a condition in $\dot{\QQQ}_\alpha$.  
   Since the components of the pair in \eqref{eq_PairNamed} are given by the union of the left and right coordinates of $\dot{c}_{g,\alpha}$, the fact that the ordering of $\dot{\QQQ}_\alpha$ is given reversed inclusion now implies that the condition $p(f_0,g) \restriction \alpha$ forces $p(f_0,g)(\alpha)$ to be stronger than every condition in $\dot{c}_{g,\alpha}$. 
     Since the validity of $\Phi_\alpha$ implies that   $p(f_0,g) \restriction \alpha$ is stronger than every condition in $g_\alpha$, it follows that $$p(f_0,g) \restriction (\alpha+1) ~ = ~ (p(f_0,g) \restriction\alpha)^\frown(\alpha,\mathsf{pair}_{\PPP_\alpha}(\dot{h}_{g,\alpha},\dot{X}_{g,\alpha}))$$ is stronger than every condition in $g_{\alpha+1}$.  

 Furthermore, by Claim \ref{clm_ClaimRE_WhatPForces}, there is an $x_\alpha \in\VV$ such that $$p(f_0,g) \restriction \alpha ~ \Vdash_{\PPP_\alpha} ~ \anf{\check{x}_\alpha = p(f_0,g)(\alpha)}.$$  Since $\Phi_\alpha$ holds, we know that $p(f_0,g) \restriction \alpha$ is flat. Let $\seq{x_\ell}{\ell \in W \cap\alpha}$ witness its flatness.  Then the sequence $\seq{x_\ell}{\ell \in W \cap(\alpha+1)}$ witnesses the flatness of $p(f_0,g) \restriction (\alpha+1)$. 
In summary, $p(f_0,g) \restriction (\alpha+1)$ is a flat condition below every member of $g_{\alpha+1}$, contradicting the fact that $\Phi_{\alpha+1}$ fails. 
\end{proof}

 Part \eqref{item_WhatStrongerConditionForces} of the lemma now follows immediately from Claim \ref{clm_SomeExtensionForces}, and part \eqref{item_IncSeqRangeDoth} of Claim \ref{clm_ClaimRE_WhatPForces}.  
\end{proof}

The above results now allow us to prove the following key lemma.

\begin{lemma}\label{lem_P_Jensen_complete}
  The poset $\PPP$ is $\kappa^+$-Jensen complete, and the flat conditions are dense in $\PPP$.
\end{lemma}

Before we prove this result, we make a couple of remarks.

\begin{remark}
 In {\cite[Claim 3.13]{MR1877015}}, a weaker version of Lemma \ref{lem_P_Jensen_complete}, stating that $\PPP$ is \emph{$\kappa$-proper}, was proven. This concept was defined in {\cite[Definition 3.4]{MR1877015}} and only makes sense under the assumption that $\kappa^{{<}\kappa} = \kappa$. It, in particular, implies that the given poset is ${<}\kappa^+$-distributive.

In the non-$\GCH$ setting, in particular, when we do not assume $\kappa^{{<}\kappa} = \kappa$, perhaps the most natural analogue of $\kappa$-properness is our notion of $\text{IA}_\kappa$-proper (Defininition \ref{def_ProperForIA}). 
 In fact, changing just a few words in the proof of {\cite[Claim 3.13]{MR1877015}} would suffice to prove that (even without $\GCH$) the poset $\PPP$ is proper for $\text{IA}_\kappa$ and is ${<}\kappa^+$-distributive.   However, that conclusion does not suffice for applications in our main theorems, since, for example, $\text{IA}_{\omega_1}$-properness, even together with ${<}\omega_2$-strategic closure, does not guarantee preservation of the \emph{Proper Forcing Axiom}.\footnote{E.g.\ if $2^{\omega_1} = \omega_2$, one can code $\text{IA}_{\omega_1} \cap \Poti{\HH{\omega_2}}{\omega_2}$ as a stationary subset $S$ of $S^2_1$.  Then, shooting an $\omega_1$-club through $S$ with initial segments is $\text{IA}_{\omega_1}$-totally proper, but kills $\PFA$.}

We seem to need the stronger property of $\kappa^+$-Jensen completeness (i.e.\ ${<}\kappa^+$-directed closure), which we prove in Lemma \ref{lem_P_Jensen_complete}.  This requires some reorganization and strengthening of the argument of {\cite[Claim 3.13]{MR1877015}}, but the main ideas of the proof of Lemma \ref{lem_P_Jensen_complete} are very similar to the proof of {\cite[Claim 3.13]{MR1877015}}.
\end{remark}

\begin{remark}
Iterations using ${<}\kappa^+$-support, where each iterand is ${<}\kappa^+$-directed closed, are themselves ${<}\kappa^+$-directed closed.  However, this fact seems to \emph{not} be applicable to the iteration $\PPP_\varepsilon$ constructed in Definition \ref{definition:Iteration}. 
 That is, it is not clear if, say, the first non-trivial poset used of the form $\PPP(S,G_0)$ is equivalent to a ${<}\kappa^+$-directed closed from the point of view of $V[G_0]$ (and we suspect it is not, in general).  The key to Lemma \ref{lem_P_Jensen_complete} (and to the analogous, but weaker {\cite[Claim 3.13]{MR1877015}}) is the flexibility in having $G_0$ not be decided yet.
\end{remark}

\begin{proof}[Proof of  Lemma \ref{lem_P_Jensen_complete}]
  First, we check $\kappa^+$-Jensen completeness.  Since each iterand is ${<}\kappa$-closed and the iteration uses $\kappa$-sized supports, the entire iteration is ${<}\kappa$-closed.  So by Corollary \ref{cor_MainCor}, to show that $\PPP$ is $\kappa^+$-Jensen complete, it suffices to show that whenever 
  \begin{itemize}
   \item $W \prec (\HH{\theta},\in,\PPP)$ with $\betrag{W}=\kappa$ and $W \cap \kappa^+ \in \kappa^+$, and 

   \item $g \subseteq W \cap \PPP$ is a $(W,\PPP)$-generic filter,
  \end{itemize}
  then $g$ has a lower bound in $\PPP$. 
     So fix such a filter $g$ for the remainder of the proof.  Given $\alpha \in W \cap [0,\varepsilon]$, define $g_\alpha$ as in Definition \ref{def_ComponentNotation_NoTails}. 
 Set $\delta = W \cap \kappa^+$.  We consider two cases: 

 \smallskip
 
 \paragraph{\textbf{Case 1:}  $\cof{\delta}< \kappa$.}  Set $f_0=\bigcup g_0$, and consider the function $p(f_0,g)$ from Definition \ref{def_ComponentNotation_NoTails}. We claim that $p(f_0,g)$ is flat condition and lies below all members of $g$.  
  Assume not. Then, by Lemma \ref{lem_KeyBackgroundLemma_NoTails}, there is an $\alpha \in W \cap [1,\varepsilon)$ such that $p(f_0,g) \restriction \alpha$ is a condition below all conditions in $g_\alpha$, and there is some $q_\alpha \leq_{\PPP_\alpha} p(f_0,g) \restriction \alpha$ in $\PPP_\alpha$ that forces all the statements in part \eqref{item_WhatStrongerConditionForces} of Lemma \ref{lem_KeyBackgroundLemma_NoTails} to hold.  In particular, by part \eqref{item_CofIsKappa}, we know that $\cof{\delta} =\cof{W \cap \kappa^+}=\kappa$, contrary to our case. 

\smallskip

\paragraph{\textbf{Case 2}:  $\cof{\delta}= \kappa$.} Since $\betrag{W}=\kappa < \delta$, we can fix $\map{t}{\delta}{\delta}$ such that $t \restriction \kappa$ is not an element of $W$. Define $$ f_0 ~ = ~  (\textstyle{\bigcup g_0}) \cup \{(\delta,t)\}.$$ 
 Given $\gamma \in \delta \cap S^{\kappa^+}_\kappa$, we have $t \restriction \gamma \notin W$ and $(\bigcup g_0)(\gamma) \in W$. In particular, we have $t \restriction \gamma \neq (\bigcup g_0)(\gamma)$ for all $\gamma \in \delta \cap S^{\kappa^+}_\kappa$. Since $\cof{\delta} = \kappa$, this shows that  $f_0$ is a condition in $\QQQ_0$ that extends $\bigcup g_0$. 

Let $p(f_0,g)$ be the function  defined in Definition \ref{def_ComponentNotation_NoTails}. We claim that $p(f_0,g)$ is a flat condition that lies below every element of $g$.  Assume not.  Then, by Lemma \ref{lem_KeyBackgroundLemma_NoTails}, there is an $\alpha \in W \cap [1,\varepsilon)$ such that $p(f_0,g) \restriction \alpha$ is a condition in $\PPP_\alpha$, and that, by part \eqref{item_EveryProperIS} of that lemma, there is some condition $q \leq_{\PPP_\alpha} p(f_0,g) \restriction \alpha$ such that $q$ forces that every proper initial segment of $(\bigcup \dot{G}_0)(\check{\delta})$ is an element of $W$.  But the 0th component of $p(f_0,g) \restriction \alpha$, and  
 hence of $q$, extends the function $f_0$, and therefore  $$q(0)\Vdash_{\PPP_0}\anf{(\textstyle{\bigcup\dot{G}_0})(\check{\delta}) = \check{t}}.$$  In particular, every proper initial segment of $t$ is an element of $W$, contrary to our choice of $t$.

\smallskip

This completes the proof of $\kappa^+$-Jensen completeness.  To see that the flat conditions are dense in $\PPP$, let $p_0$ be any condition in $\PPP$.  Fix  $W \prec (\HH{\theta},\in,\PPP,p_0)$ such that $\betrag{W}=\kappa \subseteq W$ and $W \in \text{IA}_\kappa$. 
 By Lemma \ref{lem_ClosureYieldsGenerics2} and the ${<}\kappa$-closure of $\PPP$, there exists a $(W,\PPP)$-generic filter $g$ such that $p_0 \in g$.  
 Note that $W \in \text{IA}_\kappa$ implies that $\cof{W \cap \kappa^+} = \kappa$. This shows that we can repeat the argument from the above Case 2, define $f_0$ as above and  conclude that the function $p(f_0,g)$ is a flat condition that is below every member of $g$ and therefore also below $p_0$.  
\end{proof}

Lemma \ref{lem_P_Jensen_complete} and Corollary \ref{cor_MainCor} now immediately yield the following corollary:

\begin{corollary}\label{cor_P_distributive}
 The poset $\PPP$ is forcing equivalent to a ${<}\kappa^+$-directed closed forcing.  In particular, it adds no new sets of size $\kappa$, and, in the case $\kappa = \omega_1$, it preserves all standard forcing axioms, such as $\MM^{++}$. \qed
\end{corollary}

Remember that the order $o(p)$ of a condition in a poset of the form $\PPP(S,G_0)$ was defined in part \eqref{defi_OrderOfCondition} of Definition \ref{def_PosetAddFunction}.

\begin{lemma}\label{lem_FlatImpliesBounded}
If $p$ is a flat condition in $\PPP$, then there exists $\beta < \kappa^+$ with the property that $$p\restriction\alpha\Vdash_{\PPP_\alpha}\anf{o(p(\alpha))\leq\check{\beta}}$$ holds for all $1\leq\alpha\in\supp{p}$.  
\end{lemma}

\begin{proof}
 Let $\seq{x_\alpha}{\alpha\in\supp{p}}$ be a sequence witnessing the flatness of $p$. 
 For each $\alpha\in\supp{p}$, pick $\beta_\alpha < \kappa^+$ such that $\beta_\alpha$ is not in the transitive closure of $x_\alpha$.  Since $\betrag{\supp{p}} \leq \kappa$, we know that $$\beta ~ = ~ \sup\Set{\beta_\alpha}{\alpha\in\supp{p}} ~ < ~ \kappa^+$$ has the desired properties.  
\end{proof}

Our next task is to prove that tails of the iteration behave nicely.  But first we need \emph{tail} versions of Definition \ref{def_ComponentNotation_NoTails} and Lemma \ref{lem_KeyBackgroundLemma_NoTails}.  
 Note that in Definition \ref{def_ComponentNotation_Tails} below, 
  since $\alpha_0 \geq 1$, the entire filter $G_0$ has already been determined. So unlike Definition \ref{def_ComponentNotation_NoTails}, the candidate for a condition below $g$ will not involve any $f_0$.

\begin{definition}\label{def_ComponentNotation_Tails}
 Suppose that $\alpha_0 \in [1,\varepsilon)$ and $G_{\alpha_0}$ is  $\PPP_{\alpha_0}$-generic over $\VV$.  
  Working in $\VV[G_{\alpha_0}]$, suppose that  $W  \prec (\HH{\theta}[G_{\alpha_0}], \in, \PPP/G_{\alpha_0})$ with $\betrag{W}=\kappa \subseteq W$, and $g \subseteq W \cap \PPP/G_{\alpha_0}$ is a $(W,\PPP/G_{\alpha_0})$-generic filter.  
   For each $\alpha \in W \cap [\alpha_0,\varepsilon)$, define $\PPP_\alpha/G_{\alpha_0}$-names 
    $\dot{c}_{g,\alpha}$, $\dot{h}_{g,\alpha}$ and $\dot{X}_{g,\alpha}$ analogously to Definition \ref{def_ComponentNotation_NoTails},
 and define a function $p(g)$ with domain $W \cap [\alpha_0,\varepsilon)$ by setting $$p(g) ~ = ~ \seq{\mathsf{pair}_{\PPP_\alpha/{G_{\alpha_0}}}(\dot{h}_{g,\alpha},\dot{X}_{g,\alpha})}{i \in W \cap [\alpha_0,\varepsilon)}.$$ 
\end{definition}

We now also have a \emph{tail} variant of Lemma \ref{lem_KeyBackgroundLemma_NoTails}:

\begin{lemma}\label{lem_KeyBackgroundLemma_Tails}
 Suppose $\alpha_0 \in [1,\varepsilon)$ and $G_{\alpha_0}$ is $\PPP_{\alpha_0}$-generic over $\VV$. 
  Work in $V[G_{i_0}]$ and suppose $W$ and $g$ are as in Definition \ref{def_ComponentNotation_Tails}. Given $\alpha \in [\alpha_0,\varepsilon]$, set $$g_\alpha ~ = ~  \Set{p\restriction\alpha}{p\in g}.$$ Then one of the following statements holds:  
\begin{enumerate}
 \item\label{item_GoodCase_Tails} $p(g)$ is a flat condition in $\PPP/G_{\alpha_0}$ that extends every element of $g$. 
 
 \item\label{item_BadCase_Tails} There is an $\alpha \in W \cap [\alpha_0,\varepsilon)$ such that the following statements hold: 
 \begin{enumerate}
   \item   $p(g) \restriction \alpha$ is a flat condition in $\PPP_\alpha/G_{\alpha_0}$ that is stronger than every element of $g_\alpha$. 
   
   \item\label{item_NoNewKappaSeq_Tails} If $\tau \in W$ is a $\PPP_\alpha/G_{\alpha_0}$-name for a function from $\kappa$ to the ordinals, then  $$p(g) \restriction \alpha\Vdash_{\PPP_\alpha/G_{\alpha_0}}\anf{\tau\in\check{W}}.$$  

   \item\label{item_WhatStrongerConditionForces_Tails} There is a condition $q$ in $\PPP_\alpha/G_{\alpha_0}$ below $p(g)\restriction\alpha$ with the property that the following statements hold true in $\VV[G_{\alpha_0},G]$, whenever $G$ is $\PPP_\alpha/G_{\alpha_0}$-generic over $\VV[G_{\alpha_0}]$ with $q\in G$:  
   \begin{enumerate}
    \item\label{item_CofIsKappa_Tails} $\cof{W \cap \kappa^+} = \kappa$. 
    
    \item\label{item_EveryProperIS_Tails} Every proper initial segment of $(\bigcup G_0)(W \cap \kappa^+)$ is an element of $W$, and is an $\dot{S}_\alpha^G$-node. 
    
  \item\label{item_NoProperInitial_Tails} No proper initial segment of $(\bigcup G_0)(W \cap \kappa^+)$ is an element of $\ran{\dot{X}_{g,\alpha}^G}$.
   \end{enumerate}
 \end{enumerate}
 \end{enumerate} 
\end{lemma}

\begin{proof}
 The proof is almost identical to the proof of Lemma \ref{lem_KeyBackgroundLemma_NoTails}, except we work in $\VV[G_{\alpha_0}]$ instead of $\VV$.  We leave the details to the reader.  
\end{proof}

\begin{lemma}\label{lem_TailsLessKappaClosed}
 If $\alpha  < \varepsilon$ and $G$ is $\PPP_\alpha$-generic over $\VV$, then the tail of the iteration $\PPP/G$ is ${<}\kappa$-closed in $\VV[G]$.  
\end{lemma}

\begin{proof}
 Let $\seq{q_\xi}{\xi<\mu}$ be a descending sequence with $\mu<\kappa$ in $\PPP/G$ in $\VV[G]$. 
  Since  $\PPP/G\subseteq\PPP$ and Lemma \ref{lem_P_Jensen_complete}  shows that $\PPP_\alpha$ is ${<}\kappa$-closed in $\VV$, this sequence is an element of $\VV$. Let $\dot{G}$ denote the canonical $\PPP_\alpha$-name for the generic filter in $\VV$. 
  Fix a condition $p$ in $G$ such that $$p\Vdash_{\PPP_\alpha}\anf{\textit{Every condition in $\dot{G}$ is compatible with $\check{q}_\xi$ in $\check{\PPP}$}}$$ holds in $\VV$ for all $\xi<\mu$. Work in $\VV$. Given $\xi<\mu$, a standard density argument now  shows that $$p\Vdash_{\PPP_\alpha}\anf{\check{q}_\xi\restriction\check{\alpha}\in\dot{G}}$$ and the separativity of $\PPP_\alpha$ allows us to conclude that $p\leq_{\PPP_\alpha}q\restriction\alpha$ holds. 
  
 Fix a condition $r$ below $p$ in $\PPP_\alpha$ and set  $$r_\xi ~ = ~ r^\frown (q_\xi \restriction [\alpha, \varepsilon))$$ for all $\xi<\mu$. Then $\seq{r_\xi}{\xi<\mu}$ is a descending sequence of conditions in $\PPP$, and, by the proof of Lemma \ref{lem_P_Jensen_complete}, this sequence has a lower bound $r_\mu$ in $\PPP$. Then $r_\mu\restriction\alpha\leq_{\PPP_\alpha}r$ and $r_\mu\leq_\PPP q_\xi$ for all $\xi<\mu$. 
    
  By genericity, we can now find a condition $q$ in $\PPP$ with the property that $q\restriction\alpha\in G$ and $q\leq_\PPP q_\xi$ for all $\xi<\mu$. But then we can conclude that $q$ is a condition in $\PPP/G$ in $\VV[G]$ with  $q\leq_{\PPP/G}q_\xi$ for all $\xi<\mu$. 
 \end{proof}

The proof of the following lemma is similar to the proof of {\cite[Claim 3.14]{MR1877015}}, but there are some subtle differences since we do not assume that $\kappa^{{<}\kappa} = \kappa$. Roughly, we replace their use of \emph{$\kappa$-properness} with $\text{IA}_\kappa$-properness (Definition \ref{def_ProperForIA}) and verify that the argument still goes through.

\begin{lemma}\label{lem_TailsProperIA}
  If $\alpha_0< \varepsilon$ and $G$ is $\PPP_{\alpha_0}$-generic over $\VV$, then the tail of the iteration $\PPP/G$ is $\text{IA}_\kappa$-totally proper in $\VV[G]$.
\end{lemma}

\begin{proof}
 For $\alpha_0 = 0$, the statement of the lemma follows immediately from Lemmas \ref{lem_ClosureYieldsGenerics2} and \ref{lem_P_Jensen_complete}. 
 Therefore, we  from now on assume that $1\leq\alpha_0<\varepsilon$.  

Let $G_{\alpha_0}$ be $\PPP_{\alpha_0}$-generic over $\VV$ and work in $\VV[G_{\alpha_0}]$. Let $\theta$ be a sufficiently large regular cardinal,  let $\lhd$ be a well-ordering of $\HH{\theta}=\HH{\theta}^\VV[G_{\alpha_0}]$, let $$W\prec(\HH{\theta},\in,\alpha_0,\PPP/G_{\alpha_0},\lhd)$$ with $W\in\text{IA}_\kappa$, and let $p_0$ be a condition in $W \cap(\PPP/G_{\alpha_0})$. In the following, we will find a $(W,\PPP/G_{\alpha_0})$-total master condition below $p_0$. Define $$t ~ = ~ (\textstyle{\bigcup G_0})(W \cap \kappa^+),$$ which is well-defined because $W \in \text{IA}_\kappa$ implies that $\cof{W \cap \kappa^+}=\kappa$. 

 \smallskip
 
 \paragraph{\textbf{Case 1:} \emph{There exists $\zeta \in W \cap \kappa^+$ with $t \restriction \zeta \notin W$}.} 
  Since Lemma \ref{lem_TailsLessKappaClosed} implies that $\PPP/G_{\alpha_0}$ is ${<}\kappa$-closed, we can apply Lemma \ref{lem_ClosureYieldsGenerics2} to find a $(W,\PPP/G_{\alpha_0})$-generic filter $g$ that includes $p_0$. 
   Let $p(g)$ be the function defined in Definition \ref{def_ComponentNotation_Tails} and assume, towards a contradiction, that $p(g)$ is not a condition in $\PPP/G_{\alpha_0}$ that is stronger than every element of $g$. 
   Then, by part \eqref{item_EveryProperIS_Tails} of Lemma \ref{lem_KeyBackgroundLemma_Tails}, there is an $\alpha \in W \cap [\alpha_0,\varepsilon)$ and some condition in $\PPP_\alpha/G_{\alpha_0}$ below $p(g)\restriction\alpha$ forcing that every proper initial segment of $(\bigcup G_0)(W \cap \kappa^+)$ is an element of $W$.  But this implies that every proper initial segment of $t$ is an element of $W$, contrary to our case. 
   This allows us to conclude that $p(g)$ is a $(W,\PPP/G_{\alpha_0})$-total master condition below $p_0$. 
    
 \smallskip
 
 \paragraph{\textbf{Case 2:} \emph{If $\zeta\in W\cap\kappa^+$, then $t\restriction\zeta\in W$}.}
  Since $W \in \text{IA}_\kappa$, there is a sequence $$\vec{D} ~ = ~ \seq{D_\xi}{\xi < \kappa}$$  listing all open dense subsets of $\PPP/G_{\alpha_0}$ that are elements of $W$, and such that every proper initial segment of $\vec{D}$ is an element of $W$. 
 Recursively define a descending sequence $\vec{p}=\seq{p_\xi}{\xi < \kappa}$ of conditions in $\PPP/G_{\alpha_0}$ as follows: 
\begin{itemize}
 \item Given $\xi<\kappa$, let $p^*_\xi$ be the $\lhd$-least flat condition in $\PPP/G_{\alpha_0}$ below $p_\xi$ that is an element of $D_\xi$. Such a condition exists by Lemma \ref{lem_P_Jensen_complete} and the open density of $D_\xi$.  
  By Lemma \ref{lem_FlatImpliesBounded}, there exists $\beta < \kappa^+$ with $$p^*_\xi\restriction\alpha\Vdash_{\PPP_\alpha}\anf{o(p^*_\xi(\alpha))\leq\check{\beta}}$$ for all $1\leq\alpha\in\supp{p^*_\xi}$.  
  Given $\alpha\in[\alpha_0,\varepsilon)$, let $\dot{f}_\alpha$ and $\dot{Y}_\alpha$ denote the canonical $\PPP_\alpha/G_{\alpha_0}$-names with the property that $$p^*_\xi(\alpha)^G ~ = ~ (\dot{f}_\alpha^G,\dot{Y}_\alpha^G)$$ holds whenever $G$ is $(\PPP_\alpha/G_{\alpha_0})$-generic over $\VV[G_{\alpha_0}]$ with $p^*_\xi \restriction \alpha\in G$. 
  Note that, by the choice of $\beta$, for each $\alpha\in\supp{p^*_\xi}$, the condition $p^*_\xi \restriction\alpha$ forces that $\beta$ is larger than all domains of elements of the range of $\dot{f}_\alpha$, and larger than all elements in the domain of $\dot{Y}_\alpha$. In particular, if $\alpha_0\leq\alpha\in\supp{p^*_\xi}$ and $G$ is $(\PPP_\alpha/G_{\alpha_0})$-generic over $\VV[G_{\alpha_0}]$ with $p^*_\xi \restriction\alpha\in G$, then 
 \begin{equation}\label{eq_New_ith_coord}
   (\dot{f}_\alpha^G, ~ \dot{Y}_\alpha^G ~ \cup ~ \{(\beta+1, ~ t\restriction(\beta + 1))\})
 \end{equation}
is a condition in $\dot{\QQQ}_\alpha^G$ below $p^*_\xi(\alpha)^G$.\footnote{Recall Remark \ref{remark_ReqOnX} showing that, for successor ordinals, the right coordinate of a condition does not have to agree with $G_0$).}
 This shows that there is a condition $p_{\xi+1}$ below $p^*_\xi$ in $\PPP$ with $\supp{p_{\xi+1}}=\supp{p^*_\xi}$ and the property that whenever $\alpha_0\leq\alpha\in\supp{p_{\xi+1}}$ and $G$ is $(\PPP_\alpha/G_{\alpha_0})$-generic over $\VV[G_{\alpha_0}]$ with $p^*_\xi \restriction\alpha\in G$, then   $p_{\xi+1}(\alpha)^G$ is equal to the condition in \eqref{eq_New_ith_coord}.   
 
 Now, assume that $p_\xi$ is an element of $W$. Then $p^*_\xi$ is obviously definable in $(\HH{\theta}[G_{\alpha_0}], \in,\PPP, G_{\alpha_0},\lhd)$ using the parameters $p_\xi$ and $D_\xi$, which are both contained in $W$. Since $p^*_{\xi} \in W$, then $\beta$ can also be taken to be an element of $W$. Finally, the condition $p_{\xi+1}$ is definable from $p^*_\xi$, $\beta$, and $t \restriction (\beta + 1)$, all of which are elements of $W$   because of the case we are in. These arguments show that $p_\xi\in W$ implies that $p_{\xi + 1}\in W$.

 \item If $\xi < \kappa$ is a limit ordinal, then we define $p_\xi$ be the $\lhd$-least lower bound of the sequence $\seq{p_\zeta}{\zeta<\xi}$ in $\PPP$.\footnote{Such a lower bound exists by Lemma \ref{lem_TailsLessKappaClosed}}   
\end{itemize}

Note that every proper initial segment of $\vec{p}$ is an element of $W$, because $W$ contains all proper initial segments of $t$ and each proper initial segment of $\vec{p}$ is definable in the structure $(\HH{\theta}[G_{\alpha_0}],\in,\lhd,\PPP)$ using the parameter $p_0$ and some  sufficiently long\footnote{The length of this initial segment of $t$ might depend on the given initial segment of $\vec{p}$.} proper initial segment of $t$.   
 Hence, not only is each $p_{\xi+1}$ an element of $D_\xi$, but is in fact an element of $D_\xi \cap W$.  In particular, the set $\Set{p_\xi}{\xi < \kappa}$ generates a $(W,\PPP/G_{\alpha_0})$-generic filter.  Let $g$ denote this filter, and let $p(g)$ be the function defined in Definition \ref{def_ComponentNotation_Tails}.

Now, assume, towards a contradiction, that $p(g)$ is not a condition below every member of $g$.  
 Then by part \eqref{item_NoProperInitial_Tails} of Lemma \ref{lem_KeyBackgroundLemma_Tails}, there is an $\alpha \in W \cap [\alpha_0,\varepsilon)$ and a condition $q$ in $\PPP_\alpha/G_{\alpha_0}$ below $p(g)\restriction\alpha$ with the property that whenever $G$ is $(\PPP_\alpha/G_{\alpha_0})$-generic over $\VV[G_{\alpha_0}]$ with $q\in G$, then no proper initial segment of $t=(\bigcup G_0)(W \cap \kappa^+)$ is an element of $\ran{\dot{X}_{g,\alpha}^G}$. 
 Since $\alpha \in W$ and the set $\Set{p_\xi}{\xi < \kappa}$ generates the $(W,\PPP/G_{\alpha_0})$-generic filter $g$, we can find $\xi_\alpha<\kappa$ with the property that $\alpha \in \supp{p_{\xi_\alpha + 1}}$. Then $q\leq_{}p_{\xi_\alpha +1}\restriction\alpha$. 
Let $G$ be $(\PPP_\alpha/G_{\alpha_0})$-generic over $V[G_{\alpha_0}]$ with $q \in G$.  Work in $V[G_{\alpha_0},G]$. 
 Since $p_{\xi_\alpha + 1} \in g$, we know that $p_{\xi_\alpha + 1}(\alpha)^G \in \dot{c}_{g,\alpha}^G$ and hence $\dot{X}_{g,\alpha}^G$ extends the right coordinate of $p_{\xi_\alpha + 1}(\alpha)^G$. 
  By construction, the range of the right coordinate of $p_{\xi_\alpha + 1}(\alpha)^G$ contains a proper initial segment of $t$, contradicting the properties of $\alpha$ and $q$. 
 
 Again, we can  conclude that $p(g)$ is a $(W,\PPP/G_{\alpha_0})$-total master condition below the condition $p_0$. 
\end{proof}

\begin{corollary}\label{cor_TailsPreserveComplements}
  Let $G$ be $\PPP$-generic over $\VV$, let $\alpha  \in (0,\varepsilon)$, let $G_\alpha$ be the filter on $\PPP_\alpha$ induced by $G$, and let $S=\dot{S}_\alpha^{G_\alpha}$. Then $(S^{\kappa^+}_\kappa \setminus S)^{\VV[G_\alpha]}$ is stationary in $V[G]$.  
\end{corollary}

\begin{proof}
 First, if $S=\emptyset$, then the conclusion of the lemma holds trivially, because Corollary \ref{cor_P_distributive} implies that $(S^{\kappa^+}_\kappa)^{\VV[G_\alpha]}=(S^{\kappa^+}_\kappa)^{\VV[G]}$. 
  In the other case, we know that $S^{\kappa^+}_\kappa \setminus S$ contains a stationary set in $I[\kappa^+]$ in $\VV[G_\alpha]$, and hence a combination of Lemma \ref{lem_ProperForIApreservesApproach}, Corollary \ref{cor_P_distributive} and  Lemma \ref{lem_TailsProperIA} ensures that $S^{\kappa^+}_\kappa \setminus S$ remains stationary in $\VV[G]$.   
\end{proof}

\begin{lemma}\label{lem_NoCofBranch}
 If $G$ is $\PPP$-generic over $\VV$, then the tree $\calT(G_0)$ has no cofinal branches in $\VV[G]$.
\end{lemma}

Our proof of this lemma is similar to the proof of {\cite[Claim 3.15]{MR1877015}}, but we must make the following changes: 
\begin{itemize}
 \item Whereas the proof of {\cite[Claim 3.15]{MR1877015}} makes use of ${<}\kappa$-closed elementary submodels of size $\kappa$ (whose existence requires the assumption $\kappa^{{<}\kappa} = \kappa$), we instead use  elementary submodels in $\text{IA}_\kappa$. 
 
 \item   We use Corollary \ref{cor_TailsPreserveComplements} to ensure that the complement of each $\dot{S}_\alpha$ is stationary in the final model (this is used to get the right analogue of statement (8) on page 1691 of \cite{MR1877015}). 
\end{itemize}

\begin{proof}[Proof of Lemma \ref{lem_NoCofBranch}]
 Let $\dot{b}$ be a $\PPP$-name for a function from $\kappa^+$ to $\kappa^+$. Assume, towards a contradiction, that there is a condition $p$ in $\PPP$ that forces $\dot{b}$ to be a cofinal branch through $\calT(\dot{G})$.  
 Fix $W \prec (\HH{\theta},\in,\PPP,\dot{b},p)$ with   $W \in \text{IA}_\kappa$. Set $\delta_W=W\cap\kappa^+$. By the ${<}\kappa$-closure of $\PPP$ and Lemma \ref{lem_ClosureYieldsGenerics2}, there exists a $(W,\PPP)$-generic filter $g$ containing $p$.  
 By the $(W,\PPP)$-genericity of $g$, the ${<}\kappa^+$-distributivity of $\PPP$, and the fact that $\dot{b} \in W$, it follows that for every $\gamma < \delta_W$, some condition $p_\gamma$ in $g$ decides the value of $\dot{b} \restriction \gamma$.  Define 
 \begin{equation}\label{eq_DefOf_t}
  t ~ = ~ \bigcup \Set{s \in {}^{{<}\kappa^+} \kappa^+}{\exists \gamma < \delta_W ~ p_\gamma \Vdash_\PPP\anf{\check{s} = \dot{b} \restriction \check{\gamma}}}.
 \end{equation}
 Then $t$ is a function from $\delta_W$ to $\delta_W$.  Moreover, by the $(W,\PPP)$-genericity of $g$ and the fact that $\dot{b}$ is forced to be a branch through $\dot{\calT}(\dot{G}_0)$, we know that 
 \begin{equation}
  \forall \gamma \in \delta_W \cap S^{\kappa^+}_\kappa ~ \exists p \in g ~ p \Vdash_\PPP\anf{\dot{b} \restriction \gamma \neq (\textstyle{\bigcup \dot{G}_0})(\gamma)}.
 \end{equation}
Hence, if we let $g_0$ denote the 0-th component of the $(W,\PPP)$-generic filter $g$, then we have $t \restriction \gamma \neq  (\bigcup g_0)(\gamma)$ for all $\gamma < \delta_W$.  It follows that $$f_0 ~ = ~  (\textstyle{\bigcup g_0}) ~ \cup ~ \{(\delta_W, t)\}$$ is a condition in $\QQQ_0$. 

Let $p(f_0,g)$ be the function defined in Definition \ref{def_ComponentNotation_NoTails}. Then $p(f_0,g)$ is not  a condition in $\PPP$ that extends every element of $g$, because otherwise it would force that $\dot{b} \restriction \delta_W = t = (\bigcup \dot{G}_0) (\delta_W)$ and hence it would also force that  $\dot{b} \restriction \delta_W \notin \dot{\mathcal{T}}(\dot{G}_0)$, contradicting our assumptions on $p$.

 In this situation, Lemma \ref{lem_KeyBackgroundLemma_NoTails} yields an $\alpha \in W \cap [1,\varepsilon)$ such that $p(f_0,g) \restriction \alpha$ is a condition in $\PPP_\alpha$ below every member of $g_\alpha$ and there is a condition $q$ below $p(f_0,g) \restriction \alpha$ in $\PPP_\alpha$ with the property that whenever $G$ is $\PPP_\alpha$-generic over $\VV$ with $q\in G$, then every proper initial segment of $(\bigcup G_0)(\delta_W)$ is an element of $W$, and is an $\dot{S}_\alpha^G$-node. 
Since the 0-th coordinate of $q$ extends $f_0$, we know that 
\begin{equation}\label{eq_qprimei_forces_dotf_t}
 q \Vdash_{\PPP_\alpha}\anf{(\textstyle{\bigcup\dot{G}_0})(\check{\delta}_W) = \check{t}}. 
\end{equation}

Let $H_W$ denote the transitive collapse of $W$, and let $\map{\sigma}{H_W}{W \prec \HH{\theta}}$ denote the inverse of the transitive collapsing map of $W$.  
 Set $\bar{g}=\sigma^{{-}1}[g]$, $\bar{g}_\alpha=\sigma^{{-}1}[g_\alpha]$, $\bar{\PPP}=\sigma^{{-}1}(\PPP)$ and $\bar{\PPP}_\alpha=\sigma^{{-}1}(\PPP_\alpha)$.

Let $G$ be $\PPP_\alpha$-generic over $\VV$ with $q \in G$.  Since $q$ extends every element of the $(W,\PPP_\alpha)$-generic filter $g_\alpha$, it follows that $G \cap W = g_\alpha$, $\bar{g}_\alpha$ is $\bar{\PPP}_\alpha$-generic over $H_W$, and the map $\sigma$ can be lifted to an elementary $$\map{\hat{\sigma}}{H_W[\bar{g}_\alpha]}{\HH{\theta}[G]}$$ by setting $\hat{\sigma}(\bar{\tau}^{\bar{g}_\alpha})=(\sigma(\bar{\tau}))^G $ for all $\bar{\PPP}_\alpha$-names $\bar{\tau}$ in $H_W$. 
 Moreover, we know that $\bar{g}$ is $\bar{\PPP}$-generic over $H_W$ and the function $$\map{\bar{b}= \sigma^{{-}1}(\dot{b})^{\bar{g}}}{\delta_W}{\delta_W}$$  is an element of $H_W[\bar{g}]$, but not necessarily an element of its inner model $H_W[\bar{g}_\alpha]$.

\begin{claim}\label{clm_t_barb}
$t = \bar{b}$. 
\end{claim}

\begin{proof}[Proof of Claim \ref{clm_t_barb}] 
 Fix $\gamma < \delta_W$, and set $s= t \restriction \gamma$. By earlier remarks, we know that $s \in W$ and, by the definition of $t$ in \eqref{eq_DefOf_t}, there is $p \in g\subseteq W$ with $p \Vdash_{\PPP} \anf{\dot{b} \restriction \check{\gamma} = \check{s}}$. 
 We now know that $p$, $\dot{b}$, $\gamma$, and $s$ are elements of $W = \ran{\sigma}$, and since $\crit{\sigma} = \delta_W$, it follows that $\sigma$ fixes $\gamma$ and $s$.  The elementarity of $\sigma$ now implies that $$\sigma^{{-}1}(p) \Vdash_{\bar{\PPP}}\anf{\sigma^{{-}1}(\dot{b}) \restriction \check{\gamma} = \check{s}}$$ holds in $H_W$. Moreover, since $p \in g$, we have $\sigma^{{-}1}(p) \in \bar{g}$ and hence $\bar{b} \restriction \gamma = s$ holds in $H_W[\bar{g}]$.
\end{proof}

Let $\bar{S}=\sigma^{{-}1}(\dot{S}_\alpha)^{\bar{g}_\alpha}$. 
 By Corollary \ref{cor_TailsPreserveComplements} and the elementarity of $\map{\sigma}{H_W}{\HH{\theta}}$, we know that $(S^{\delta_W}_\kappa\setminus \bar{S})^{H_W[\bar{g}_\alpha]}$ remains stationary when going from $H_W[\bar{g}_\alpha]$ to $H_W[\bar{g}]$. 
  Since $\bar{b}$ maps from $\delta_W$ to $\delta_W$ and $(S^{\delta_W}_\kappa\setminus \bar{S})^{H_W[\bar{g}_\alpha]}$ is stationary in $H_W[\bar{g}]$, there is an $\ell < \delta_W$ such that the following statements hold in $H_W[\bar{g}]$: 
 \begin{itemize}
  \item $\cof{\ell} = \kappa$. 
 
  \item $\ell$ is closed under $\bar{b}$. 
  
 \item $\ell \notin \bar{S}$. 
\end{itemize}

 Then $b_*= \bar{b} \restriction \ell$ maps from $\ell$ to $\ell$, and, by the ${<}\delta_W$-distributivity of $\bar{\PPP}$ over $H_W$, we know that $b_*$ is an element of $H_W$. 
 Moreover, since $\crit{\hat{\sigma}} = \delta_W$, we have $\hat{\sigma}(b_*) = b_*$.  
 In addition, since $\ell = \dom{b_*} \in (S^{\delta_W}_\kappa\setminus \bar{S})^{H_W[\bar{g}_\alpha]}$ and $\ell$ is closed under $b^*$, we can conclude that $H_W[\bar{g}_\alpha]$ believes that $b_*$ is not a $\bar{S}$-node.\footnote{Recall from Definition \ref{def_PosetAddFunction} that a function $s$ is an $S$-node if no element of $S^{\kappa^+}_\kappa \setminus S$ is closed under $s$} 
Then the elementarity of $\map{\hat{\sigma}}{H_W[\bar{g}_\alpha]}{\HH{\theta}[G]}$ implies that $\hat{\sigma}(b_*) = b_*$ is not a $\dot{S}_\alpha^G$-node in  $\HH{\theta}[G]$. 
  By Claim \ref{clm_t_barb}, we now have $b_* = t \restriction \ell$.  So $t \restriction \ell$ is not a $\dot{S}_\alpha^G$-node, contradicting our earlier arguments. 
\end{proof}

 We are now ready to complete the proof of the main technical result of this paper.

\begin{proof}[Proof of Theorem \ref{thm_PropertiesOfHytRautPoset}]
 Let $\kappa$ be an infinite regular cardinal and let $\PPP=\PPP_\varepsilon$ be the poset constructed  in Definition \ref{definition:Iteration}. Then Lemma \ref{lem_P_Jensen_complete} shows that part \eqref{item_Jensen_complete} of the theorem holds.

 Next, we prove part \eqref{item_ChainCondition} of the theorem.  We will prove that the poset $\PPP$ is \emph{$(2^\kappa)^+$-stationarily layered} (see {\cite[Definition 29]{MR3911105}}) in $\VV$, which, by {\cite[Lemma 4]{MR3911105}}, implies that $\PPP$ is $(2^\kappa)^+$-Knaster.  A poset $\RRR$ is $\lambda$-stationarily layered if for some sufficiently large regular cardinal $\theta$, there are stationarily-many $M \in \POTI{\HH{\theta}}{\lambda}$ such that $M \cap \RRR$ is a regular suborder of $\RRR$. Equivalently, we can demand that every condition $p$ in $\RRR$  has a \emph{reduction} into $M \cap \RRR$, i.e.\ there exists $q \in M \cap \RRR$ such that all extensions of $q$ in $M \cap \RRR$ are compatible with $p$ in $\PPP$.

For all sufficiently large regular cardinals $\theta$, the set 
 $$ R ~ = ~ \Set{M \in \POTI{\HH{\theta}}{(2^\kappa)^+}}{M\prec\HH{\theta}, ~{}^{\kappa} M \subseteq M, ~ \PPP\in M}$$ is stationary in $\Poti{\HH{\theta}}{(2^\kappa)^+}$.  
   We prove that $R$ witnesses the $(2^\kappa)^+$-stationary layeredness of $\PPP$.  Fix $M \in R$ and a condition $p$ in $\PPP$.  By the density of flat conditions in $\PPP$, we we may assume that there exists a sequence $\seq{x_\alpha}{\alpha\in\supp{p}}$  that witnesses that $p$ is flat. 
   Furthermore, we may assume that $$p(\alpha)=\check{x}_\alpha ~ \Longleftrightarrow ~ p(\alpha)\neq\check{\emptyset} ~ \Longleftrightarrow ~ {p\restriction\alpha}\Vdash_{\PPP_\alpha}\anf{\check{x}_\alpha \in \dot{\QQQ}_\alpha}$$ holds for all $\alpha\in\supp{p}$, because redefining $p$ in this way results in a condition equivalent to $p$. 
      Set $s= M \cap \supp{p}$ and define $q= p \restriction s$.

  \begin{claim}\label{claim:ReductionIntoSubmodel}
   The condition $q$ is a reduction of $p$ into $M\cap\PPP$.  
  \end{claim}

 \begin{proof}[Proof of Claim \ref{claim:ReductionIntoSubmodel}]
  First, we verify that $q$ is an element of $M$.   Since we have $x_\alpha \in\HH{\kappa^+} \subseteq M$ for all $\alpha\in\supp{p}$, the closure properties of $M$ imply that  the sequence $\seq{x_\alpha}{\alpha \in s}$ is an element of $M$. Since the condition $q$ is definable from the sequence $\seq{x_\alpha}{\alpha\in s}$, it follows that $q$  is also an element of $M$.

 Next, assume $r$ is a condition in $M\cap\PPP$ below $q$. 
 Let $p \wedge r$ denote the natural amalgamation of $p$ and $r$, i.e.\ we have $(p \wedge r)(\beta)= r(\beta)$ for all $\beta \in \supp{r}$, and $(p \wedge r)(j) = p(\beta)$ for all $\beta \in \supp{p} \setminus \supp{r}$.  Since $p \wedge r$ is clearly a function whose support has size at most $\kappa$, it is a condition in $\PPP$. 
 We verify that $p \wedge r$  is below both $p$ and $r$ in $\PPP$ by checking inductively that $(p \wedge r) \restriction \beta$ is below both $p \restriction\beta$ and $p \restriction\beta$ for all $\beta\leq\varepsilon$.  Suppose this statement holds at all $\alpha < \beta\leq\varepsilon$. 
  Clearly, if $\beta$ is a limit ordinal, then it holds at $\beta$ as well. Hence, we may assume that $\beta =\alpha+1$ and that $(p \wedge r) \restriction \alpha$ lies below both $p \restriction \alpha$ and $r \restriction \alpha$.   
  If $\alpha$ is not in the support of either $p$ or $r$, then the above statement trivially holds at $\beta$ as well. Hence, we have to consider the following  two cases: 
  
  \smallskip
  
 \paragraph{\textbf{Case 1:} $\alpha \in \supp{r}$.} 
  By definition of the condition $p \wedge r$, we have $(p \wedge r)(\alpha) = r(\alpha)$ and $r \leq_\PPP q$ implies that $r \restriction\alpha\Vdash_{\PPP_\alpha}\anf{r(\alpha)\leq_{\dot{\QQQ}_\alpha}q(\alpha)}$.  Since $(p \wedge r) \restriction \alpha \leq_{\PPP_\alpha} r \restriction \alpha$ by our induction hypothesis, this shows that 
 \begin{equation}\label{eq_r_wedge_p}
  (p \wedge r) \restriction \alpha \Vdash_{\PPP_\alpha}\anf{(p \wedge r)(\alpha) \leq_{\dot{\QQQ}_\alpha} q(\alpha)}. 
 \end{equation} 
 Moreover,  since $r \in M$ and $\betrag{\supp{r}} \leq \kappa \subseteq M$, we have $\alpha\in\supp{r} \subseteq M$. In particular,  if $\alpha\in\supp{p}$, then $\alpha \in s$ and hence $p(\alpha) = q(\alpha)$.  In combination with \eqref{eq_r_wedge_p}, this yields 
 \begin{equation}\label{equation_layered_nontrivial_case}
  (p \wedge r) \restriction \alpha \Vdash_{\PPP_\alpha}\anf{(p \wedge r)(\alpha)  \leq_{\dot{\QQQ}_\alpha} p(\alpha)}. 
 \end{equation}
In the other case, if $\alpha \notin \supp{p}$, then \eqref{equation_layered_nontrivial_case} holds  trivially.
  
    \smallskip
  
 \paragraph{\textbf{Case 2:} $\alpha \in \supp{p}\setminus\supp{r}$.} 
 By the definition of $(p \wedge r)$, we have $(p \wedge r)(\alpha) = p(\alpha)$.  Since $\alpha \notin \text{sprt}(r)$, it follows trivially that $$(p\wedge r)\restriction\alpha\Vdash_{\PPP_\alpha}\anf{(p\wedge r)(\alpha)\leq_{\dot{\QQQ}_\alpha}p(\alpha)\leq_{\dot{\QQQ}_\alpha}r(\alpha)}.$$ 
 
 \smallskip
 
 These computations show that $q$ is a reduction of $p$ into $M \cap \PPP$.  
\end{proof}
  
This concludes the proof that the poset is $\PPP$ is $(2^\kappa)^+$-stationarily layered, and hence $(2^\kappa)^+$-Knaster.

 We now verify part \eqref{item_ExistsTreeForApproach} of the theorem. 
Let $G$ be $\PPP$-generic over $\VV$. 
 Suppose $S$ is a bistationary subset of $S^{\kappa^+}_\kappa$ in $\VV[G]$ such that $S^{\kappa^+}_\kappa \setminus S$ contains a stationary set $T$ in $I[\kappa^+]$ in $\VV[G]$.  
 Pick a club $D$ in $\kappa^+$ and a $\kappa^+$-sequence $\vec{z}$ of sets from $[\kappa^+]^{{<}\kappa}$ witnessing that $T$ is an element of $I[\kappa^+]$ in $\VV[G]$. 
 A combination of Lemma \ref{lem_P_Jensen_complete} and part \eqref{item_ChainCondition} of this theorem now shows that there is a subset $P\in\VV$ of $S^{\kappa^+}_\kappa\times 2^\kappa$ and a sequence $\seq{q_{\gamma,\xi}}{(\gamma,\xi)\in P}\in\VV$ of flat conditions in $\PPP$ such that $\dot{B}^G=S$, where $\dot{B}$ is the $\PPP$-name $\Set{(\check{\gamma},q_{\gamma,\xi})}{(\gamma,\xi)\in P}$. 
 Pick an element $s$ of the set $\calN$ defined before Definition \ref{definition:Iteration} such that $\dom{s}=P$ and, if $(\gamma,\xi)\in P$, then $\dom{s(\gamma,\xi)}=\supp{q_{\gamma,\xi}}$ and $q_{\gamma,\xi}(\ell)=\check{x}$ for all $\ell\in\dom{s(\gamma,\xi)}$ with $b(\gamma,\xi)(\ell)=x$. 
 By our assumptions on $\varepsilon$ and $b$, part \eqref{item_ChainCondition} of this theorem allows us to find $0<\alpha<\varepsilon$ with the property that $\dot{B}$ is a $\PPP_\alpha$-name, $b(\alpha)=s$ and $\vec{z},D,T\in\VV[G_\alpha]$, where $G_\alpha$ is the filter on $\PPP_\alpha$ induced by $G$.
   Clearly, the fact that $T$ is bistationary in $\VV[G]$ implies that $T$ is bistationary in $\VV[G_\alpha]$. Moreover, since $\VV[G]$ and $\VV[G_\alpha]$ have the same $\kappa$-sequences of ordinals,  every element of $D \cap T$ is also approachable with respect to $\vec{z}$ in $\VV[G_\alpha]$. 
     Hence, we know that $S^{\kappa^+}_\kappa \setminus S$ contains a stationary set in $I[\kappa^+]$ in $\VV[G_\alpha]$. 
    Since our choice of $\alpha$ ensures that $\dot{B}_\alpha=\dot{B}$, we can conclude that $\dot{S}_\alpha^{G_\alpha}=\dot{B}_\alpha^{G_\alpha}=\dot{B}^G=S$ and hence forcing with $\dot{\QQQ}_\alpha^{G_\alpha}$ over $\VV[G_\alpha]$ adds an order-preserving function from $T(S)$ to $\calT(G_0)$.

 Finally, we  prove part \eqref{item_MaximalSet} of the theorem. 
 Hence, assume that $\kappa^{{<}\kappa} \leq \kappa^+$ holds in $\VV$ and fix an enumeration $\vec{z} =  \seq{z_\xi}{\xi < \kappa^+}$ of all elements of $[\kappa^+]^{{<}\kappa}$ in $\VV$. 
  By Lemma \ref{lem_CardArith_MaximalElement}, the set $M$ of all $\gamma \in S^{\kappa^+}_\kappa$ that are approachable with respect to $\vec{z}$ is a maximal element of $I[\kappa^+] \cap \POT{S^{\kappa^+}_\kappa}$ mod $\NS{}$ in $\VV$.  Since $\PPP$ is ${<}\kappa^+$-distributive and therefore $\vec{z}$ still enumerates all of $[\kappa^+]^{<\kappa}$ in $\VV[G]$, it follows that $M$ is still the set of all $\gamma \in S^{\kappa^+}_\kappa$ that are approachable with respect to $\vec{z}$ in $V[G]$, and hence $M$ is still a maximal element  of $I[\kappa^+] \cap \POT{S^{\kappa^+}_\kappa}$ mod $\NS{}$ in $\VV[G]$. 
   Now, suppose that $S \in\VV[G]$ is bistationary in $S^{\kappa^+}_\kappa$ and $M \setminus S$ is stationary.  Since $M \in I[\kappa^+]$ and $I[\kappa^+]$ is an ideal, it follows that $M \setminus S \in I[\kappa^+]$.  So $M \setminus S$ is a stationary set in $I[\kappa^+]$.   Hence by part \eqref{item_ExistsTreeForApproach} of the theorem,  there is an order-preserving function from $T(S)$  to $\calT(G_0)$ in $V[G]$. 
\end{proof}


\section{Applications}

 We now apply Theorem \ref{thm_PropertiesOfHytRautPoset} to prove the results presented in the introduction of the paper.

\begin{corollary}\label{cor_Def_NS_MaxSet}
 Let $\kappa$ be an infinite regular cardinal satisfying $\kappa^{{<}\kappa} \leq \kappa^+$, let $\PPP$ be the partial order given by Theorem \ref{thm_PropertiesOfHytRautPoset} and let $M$ be a maximum element of $I[\kappa^+] \cap \POT{S^{\kappa^+}_\kappa}$ mod $\NS{}$. If $G$ is $\PPP$-generic over $\VV$, then the set $\NS{}\restriction M$ is $\mathbf{\Delta}_1(\HH{(2^\kappa)^+})$-definable in $\VV[G]$.  
\end{corollary}

\begin{proof}
 Work in $\VV[G]$ and let $T$ be the subtree of ${}^{{<}\kappa^+}\kappa^+$ given by Theorem \ref{thm_PropertiesOfHytRautPoset}. Then $T\subseteq{}^{{<}\kappa^+}\kappa^+\in\HH{(2^\kappa)^+}$. Define $\calS$ to be the collection of all subsets $A$ of $M$ such that either there exists a closed unbounded subset $C$ of $\kappa^+$ with $C\cap M\subseteq A$ or there exists an order-preserving function from the tree $T(S^{\kappa^+}_\kappa\setminus A)$ into the tree $T$. 
  Then the set $\calS$ is definable by a $\Sigma_1$-formula with parameters $M$, $T$ and ${}^{{<}\kappa^+}\kappa^+$. 
 
 \begin{claim}\label{claim_definability_from_technial_result}
   The set $\calS$ is equal to the collection of all subsets of $M$ that are stationary in $\kappa^+$. 
 \end{claim}
 
 \begin{proof}[Proof of Claim \ref{claim_definability_from_technial_result}]
  First, let $A\subseteq M$ be stationary in $\kappa^+$ with the property that there is no club $C$ in $\kappa^+$ with $C\cap M\subseteq A$. Since $M$ is stationary in $\kappa^+$, this shows that $A$ is bistationary in $S^{\kappa^+}_\kappa$, $M\setminus A$ is stationary, and hence Theorem \ref{thm_PropertiesOfHytRautPoset} yields an order-preserving function from $T(S^{\kappa^+}_\kappa\setminus A)$ into $T$ that witnesses that $A$ is contained in $\calS$. This argument shows that $\calS$ contains all stationary subsets of $M$. 
  
 Now, assume, towards a contradiction, that there is a non-stationary subset $A$ of $\kappa^+$ that is contained in $\calS$. Then there is an order-preserving embedding of $T(S^{\kappa^+}_\kappa\setminus A)$ into $T$ and a closed unbounded subset $C$ of $\kappa^+$ with $A\cap C=\emptyset$. But then $C\cap S^{\kappa^+}_\kappa$ is a $\kappa$-club that is a subset of $S^{\kappa^+}_\kappa\setminus A$ and, by earlier remarks, the tree $T(S^{\kappa^+}_\kappa\setminus A)$ contains a cofinal branch. But then the tree $T$ also contains a cofinal branch, a contradiction. 
 \end{proof}
 
 By the above claim, the set $\NS{}\restriction M=\POT{M}\setminus\calS$ is definable by a $\Pi_1$-formula with parameters in $\HH{(2^\kappa)^+}$. 
\end{proof}

In particular, the above corollary directly shows how the definability results of \cite{MR1877015} and \cite{MR1222536} can be derived from Theorem \ref{thm_PropertiesOfHytRautPoset}.

 \begin{corollary}
  Let $\kappa$ be an infinite regular cardinal satisfying $\kappa^{{<}\kappa} = \kappa$ and  let $\PPP$ be the poset given by Theorem \ref{thm_PropertiesOfHytRautPoset}. If $G$ is $\PPP$-generic over $\VV$, then $\NS{} \restriction S^{\kappa^+}_\kappa$ is $\mathbf{\Delta}_1(\HH{(2^\kappa)^+})$-definable in $\VV[G]$. 
 \end{corollary}
 
 \begin{proof}
 By Lemma \ref{lem_CardArith_MaximalElement}, if $\kappa^{{<}\kappa} = \kappa$ holds in $\VV$, then $S^{\kappa^+}_\kappa$ is a  maximum element of $I[\kappa^+] \cap \POT{S^{\kappa^+}_\kappa}$ mod $\NS{}$. Since forcing with $\PPP$ does not change cofinalities below $\kappa^+$, the desired conclusion directly follows from Corollary \ref{cor_Def_NS_MaxSet}.   
\end{proof}

 The following lemma establishes a connection between principles of stationary reflection and the $\Pi_1$-definability of restrictions of the non-stationary ideals that will be crucial for proofs of our main results.

\begin{lemma}\label{lemma:DefFromRefl}
 Let $S$ be a stationary subset of an uncountable regular cardinal $\delta$ and let $\calE$ be a set of stationary subsets of $S^\delta_{{>}\omega}$ with the property that for every stationary subset $A$ of $S$, there exists $E\in \calE$ such that $A$ reflects at every element of $E$. 
 If $\calE$ is definable by a $\Sigma_1$-formula with parameter $p$, then the set $\NS{}\restriction S$ is  definable by a $\Pi_1$-formula with parameters $p$, $S$ and $\HH{\delta}$. 
\end{lemma}

\begin{proof}
 Let $\calS$ denote the collection of all subsets $A$ of $S$ with the property that there exists $E\in\calE$ such that $A\cap\alpha$ is stationary in $\alpha$ for all $\alpha\in E$. By our assumptions on $\calE$, the set $\calS$ is definable by a $\Sigma_1$-formula with parameters $p$, $S$ and $\HH{\delta}$.  
  If $A\subseteq S$ is stationary in $\delta$, then our assumptions on $\calE$ ensure that $A$ is contained in $\calS$. 
  In the other direction, if $E\in\calE$ witnesses that $A$ is an element of $\calS$ and $C$ is closed unbounded in $\delta$, then there is  $\alpha\in E\cap\Lim(C)$ with $A\cap\alpha$ stationary in $\alpha$ and hence $\emptyset\neq A\cap C\cap\alpha\subseteq A\cap C$. Together, this shows that $\calS$ is equal to the collection of all subsets of $S$ that are stationary in $\delta$ and hence $\NS{}\restriction S=\POT{\delta}\setminus\calS$ is definable by a $\Pi_1$-formula with parameters $p$, $S$ and $\HH{\delta}$.   
\end{proof}

 The above lemma directly shows that strong forms of stationary reflection cause restrictions of non-stationary ideals to be $\mathbf{\Delta}_1$-definable.

\begin{corollary}
 Let $\delta$ be an uncountable regular cardinal, let $E$ be a stationary subset of $S^\delta_{{>}\omega}$ and let $S$ be a stationary subset of $\delta$ such that every stationary subset of $S$ reflects almost everywhere in $E$ (i.e. for every stationary subset $A$ of $S$, there is a closed unbounded subset $C$ of $\delta$ with the property that $A$ reflects at every element of $C\cap E$).   Then the set $\NS{}\restriction S$ is definable by a $\Pi_1$-formula with parameters $E$, $S$ and $\HH{\delta}$. 
\end{corollary}

\begin{proof}
 If we define $\calE=\Set{C\cap E}{\textit{$C$ club in $\delta$}}$, then $\calE$ is definable by a $\Sigma_1$-formula with parameter $E$ and this shows that the sets $\calE$ and $S$ satisfy the assumptions of Lemma \ref{lemma:DefFromRefl}.  
\end{proof}

 Note that a classical result of Magidor in \cite{MR683153} shows that, starting with a weakly compact cardinal, it is possible to construct a model of set theory in which every stationary subset of $S^2_0$ reflects almost everywhere in $S^2_1$. The above corollary shows that the set $\NS{}\restriction S^2_0$ is $\mathbf{\Delta}_1(\HH{\omega_3})$-definable in Magidor's model.

The next theorem will be used to derive Theorem \ref{MainTheorem1}.

  \begin{theorem}\label{AltMainTheorem1}
  Assume that 
   $2^{\omega_1}=\omega_2$, and $\theta$ is a cardinal with $\theta^{\omega_2}=\theta$. 
  Then there exists a ${<}\omega_2$-directed closed, cardinal-preserving poset $\PPP$ with the property that the following statements hold in $\VV[G]$ whenever $G$ is $\PPP$-generic over $\VV$: 
 \begin{enumerate}
  \item $2^{\omega_2}=\theta$. 

  \item\label{item:AltTheorem1item2} If for every stationary subset $A$ of $S^2_0$, there is a stationary subset $R$ of $\text{IA}_{\omega_1}$ such that $W \prec H(\omega_3)$ and $A$ reflects at $W \cap \omega_2$ for all $W \in R$, then the set $\NS{}\restriction S^2_0$ is $\mathbf{\Delta}_1(\HH{\omega_3})$-definable. 
 \end{enumerate}
 \end{theorem}
 
 \begin{proof}
  Let $G$ be $\Add{\omega_2}{\theta}$-generic over $\VV$.  Since $2^{\omega_1}=\omega_2$ holds in $\VV$, we know that $\Add{\omega_2}{\theta}$ satisfies the $\omega_3$-chain condition in $\VV$ and hence all cofinalities are preserved in $\VV[G]$.  
  Work in $\VV[G]$. Then our assumptions ensure that $2^{\omega_1}=\omega_2$ and $2^{\omega_2}=\theta=\theta^{\omega_2}$. 
   Let $\PPP$ be the poset given by Theorem \ref{thm_PropertiesOfHytRautPoset} for $\kappa=\omega_1$ and $\varepsilon=\theta$, and let $M$ be a maximum element of $I[\omega_2] \cap \POT{S^{\omega_2}_{\omega_1}}$ mod $\NS{}$, which exists due to the assumption that $2^\omega \le \omega_2$ (see Lemma \ref{lem_CardArith_MaximalElement}). 
  Then Lemma \ref{lem_GeneralizeJensenComplete2} and part \eqref{item_Jensen_complete} of Theorem \ref{thm_PropertiesOfHytRautPoset} show  that $\PPP$ is forcing equivalent to a ${<}\omega_2$-directed closed poset. 
 Moreover, since $2^{\omega_1}=\omega_2$ holds, part \eqref{item_ChainCondition} of Theorem \ref{thm_PropertiesOfHytRautPoset} shows that $\PPP$ satisfies the $\omega_3$-chain condition. 
  Finally, Lemma \ref{lem_P_Jensen_complete} shows that $\PPP$ has a dense subset of cardinality $\theta$.

  Now, let $H$ be $\PPP$-generic over $\VV[G]$ and work in $\VV[G,H]$. 
   By the above observations, we then have $2^{\omega_2}=\theta$. In addition, part \eqref{item_MaximalSet} of Theorem \ref{thm_PropertiesOfHytRautPoset} shows that $M$ is the maximum element of $I[\omega_2] \cap \POT{S^{\omega_2}_{\omega_1}}$ mod $\NS{}$. Moreover, Corollary \ref{cor_Def_NS_MaxSet} shows that $\NS{}\restriction M$ is $\mathbf{\Delta}_1(\HH{\omega_3})$-definable. 
  In the following, assume that for every stationary subset $A$ of $S^2_0$, there is a stationary subset $R$ of $\text{IA}_{\omega_1}$ such that $W \prec H(\omega_3)$ and $A$ reflects at $W \cap \omega_2$ for all $W \in R$. 
  Set $\calE=\POT{M}\setminus\NS{\omega_2}$. Then $\calE$ is definable by a $\Sigma_1$-formula with parameters in $\HH{\omega_3}$.  
  
  \begin{claim}
   For every stationary subset $A$ of $S^2_0$, there is an element $E$ of $\calE$ with the property that $A$ reflects at every element of $E$.
  \end{claim}
  
  \begin{proof}[Proof of the Claim]
   By our assumption, there is a stationary subset $R$ of $\text{IA}_{\omega_1}$ such that $W \prec H(\omega_3)$ and $A$ reflects at $W \cap \omega_2$ for every $W \in R$. 
  If we now define $E_0=\Set{W \cap \omega_2}{W \in R}$, then $E_0$ is a stationary subset of $S^{\omega_2}_{\omega_1}$. 
  Moreover,  since $2^\omega \leq \omega_2$, each $W \in R$ has (as an element) an enumeration $\vec{z} = \seq{z_\xi}{\xi<\omega_2}$ of $[\omega_2]^\omega$ and therefore the internal approachability of $W$ and the fact that $\vec{z} \in W$ imply that $W \cap \omega_2$ is approachable with respect to  $\vec{z}$.  Hence, the set $E_0$ is stationary and an element of  $I[\omega_2]$. 
  Since $M$ is the largest such element mod $\NS{}$, we have in particular that $E= E_0 \cap M$ is a stationary subset of $M$.      
 \end{proof}
  
  Using Lemma \ref{lemma:DefFromRefl}, we can now conclude that $\NS{}\restriction S^2_0$ is definable by a $\Sigma_1$-formula with parameters in $\HH{\omega_3}$.  
 \end{proof}


 \begin{proof}[Proof of Theorem \ref{MainTheorem1}]
Assume that $\FA$ holds, where $\FA$ is one of the following axioms:
\begin{itemize}
 \item $\MM^{+\mu}$, where $\mu$ is a cardinal and $0 \le \mu \le \omega_1$; or
 \item $\PFA^{+\mu}$, where $\mu$ is a cardinal and $1 \le \mu \le \omega_1$.
\end{itemize}
Let $\theta$ be a cardinal with $\theta^{\omega_2}=\theta$. 
Since $\PFA$ implies that $2^\omega=2^{\omega_1}=\omega_2$ holds (see \cite[Theorem 16.20 \& 31.23]{MR1940513}), our assumption allows us to apply Theorem \ref{AltMainTheorem1} to obtain a ${<}\omega_2$-directed closed poset with the properties listed in the conclusion of theorem.  
  Let $G$ be $\PPP$-generic over $\VV$. Then {\cite[Theorem 4.7]{cox2018forcing}} ensures that $\FA$ holds in $\VV[G]$.  Since $\FA$ holds in $\VV[G]$, there exists a stationary subset $R$ of $\text{IA}_{\omega_1}$ with the property that for all $W \in R$, we have $W \prec H(\omega_3)$ and $A$ reflects at $W \cap \omega_2$.\footnote{For the case corresponding to $\MM$, this follows by the proof of    {\cite[Theorem 13]{10.2307/1971415}}.  For the case corresponding to $\PFA^{+\mu}$ where $\mu \ge 1$, it follows from the remark on  {\cite[p. 20]{10.2307/1971415}}.  The $\omega_1$-enumerations in both proofs are easily seen to be internally approachable enumerations.}
 Then Theorem \ref{AltMainTheorem1} allows us to conclude that $\NS{}\restriction S^2_0$ is $\mathbf{\Delta}_1(\HH{\omega_3})$-definable in $\VV[G]$.  
 \end{proof}
 

The next theorem will be used to derive Theorem \ref{MainTheorem2}.
  \begin{theorem}\label{AltMainTheorem2}
  Assume that $2^\omega=\omega_1$, $2^{\omega_1}=\omega_2$, and $\theta$ is a cardinal with $\theta^{\omega_2}=\theta$. 
  Then there exists a ${<}\omega_2$-directed closed, cardinal-preserving poset $\PPP$ with the property that the following statements hold in $\VV[G]$ whenever $G$ is $\PPP$-generic over $\VV$: 
 \begin{enumerate}
  \item $2^{\omega_2}=\theta$. 

  \item\label{item:AltTheorem2item2} If every stationary subset of $S^2_0$ reflects to a point in $S^2_1$, then the set $\NS{\omega_2}$ is $\mathbf{\Delta}_1(\HH{\omega_3})$-definable. 
 \end{enumerate}
 \end{theorem}

\begin{proof}
 Let $G$ be $\Add{\omega_2}{\theta}$-generic over $\VV$, let $\PPP$ be the poset produced by an application of Theorem \ref{thm_PropertiesOfHytRautPoset} with $\kappa=\omega_1$ and $\varepsilon=\theta$ in $\VV[G]$, and let $H$ be $\PPP$-generic over $\VV[G]$. 
 As above, we  have $(2^{\omega_2})^{\VV[G,H]}=\theta$ and, since $2^\omega=\omega_1$ holds in $\VV[G]$,  part \eqref{item_CardArith_AP_fail} of Lemma \ref{lem_CardArith_MaximalElement} and part \eqref{item_MaximalSet} of Theorem \ref{thm_PropertiesOfHytRautPoset} imply that $(S^2_1)^{\VV[G]}=(S^2_1)^{\VV[G,H]}$ is a maximum element of $I[\omega_2] \cap \POT{S^2_1}$ mod $\NS{}$ in both $\VV[G]$ and $\VV[G,H]$. In particular, Corollary \ref{cor_Def_NS_MaxSet} implies  that $\NS{}\restriction S^2_1$ is $\mathbf{\Delta}_1(\HH{\omega_3})$-definable in $\VV[G,H]$. 
 
  Now, work in $\VV[G,H]$ and assume that every stationary subset of $S^2_0$ reflects to a point in $S^2_1$. 
  Then every stationary subset of $S^2_0$ reflects to stationary-many points in $S^2_1$ and we can apply Lemma \ref{lemma:DefFromRefl} with $S^2_0$ and $\POT{S^2_1}\setminus\NS{\omega_2}$ to show that $\NS{}\restriction S^2_0$ is $\mathbf{\Delta}_1(\HH{\omega_3})$-definable. Since it is easy to see that 
  $$\NS{\omega_2} ~ = ~ \Set{A\subseteq\omega_2}{A\cap S^2_0\in\NS{}\restriction S^2_0 \text{ and }   A\cap S^2_1\in\NS{}\restriction S^2_1},$$  these computations allow us to conclude that $\NS{\omega_2}$ is $\mathbf{\Delta}_1$-definable.  
\end{proof}

 \begin{proof}[Proof of Theorem \ref{MainTheorem2}]
  Assume that $2^\omega=\omega_1$, $2^{\omega_1}=\omega_2$ and  either $\FA^{{+}}(\text{$\sigma$-closed})$ or $\mathrm{SCFA}$ holds. 
  Let $\theta$ be a cardinal with $\theta^{\omega_2}=\theta$, let $\PPP$ be the poset produced by an application of Theorem   \ref{AltMainTheorem2} and let $G$ be $\PPP$-generic over $\VV$. 
 Then $(2^{\omega_2})^{\VV[G]}=\theta$ holds. Moreover, by Theorem 4.7 of \cite{cox2018forcing}, either $\FA^{{+}}(\text{$\sigma$-closed})$ or $\mathrm{SCFA}$ holds in $\VV[G]$.  In the case of $\mathrm{SCFA}$, the fact that CH also holds ensures (by {\cite[Theorem 2.7 and Observation 2.8]{fuchs_2018}}) that, in $\VV[G]$, every stationary subset of $S^2_0$ reflects in a point in $S^2_1$.\footnote{Note that the $\CH$ assumption seems to be required for this consequence of SCFA; see Fuchs {\cite{fuchs2020canonical}} for some corrections on previous literature.}  In the case where $\text{FA}^+\big( \sigma \text{-closed} \big)$ holds, the proof of Theorem 8.3 of  \cite{MR776640} ensures the same kind of stationary reflection. 
  By Theorem \ref{AltMainTheorem2}, this shows that $\NS{\omega_2}$ is $\mathbf{\Delta}_1$-definable in $\VV[G]$.   
 \end{proof}


 \bibliographystyle{plain}
 \bibliography{references}

\begin{thebibliography}{10}

\bibitem{AsperoSchindler}
David Asper\'{o} and Ralf Schindler.
\newblock Martin's {M}aximum{$^{++}$} implies {W}oodin's axiom {$(*)$}.
\newblock {\em Ann. of Math. (2)}, 193(3):793--835, 2021.

\bibitem{MR776640}
James~E. Baumgartner.
\newblock Applications of the proper forcing axiom.
\newblock In {\em Handbook of set-theoretic topology}, pages 913--959.
  North-Holland, Amsterdam, 1984.

\bibitem{MR3911105}
Sean Cox.
\newblock Layered posets and {K}unen's universal collapse.
\newblock {\em Notre Dame Journal of Formal Logic}, 60(1):27--60, 2019.

\bibitem{cox2018forcing}
Sean Cox.
\newblock Forcing axioms, approachability, and stationary set reflection.
\newblock {\em J. Symb. Log.}, 86(2):499--530, 2021.

\bibitem{MR2160657}
James Cummings.
\newblock Notes on singular cardinal combinatorics.
\newblock {\em Notre Dame J. Formal Logic}, 46(3):251--282, 2005.

\bibitem{10.2307/1971415}
Matthew Foreman, Menachem Magidor, and Saharon Shelah.
\newblock {M}artin's {M}aximum, saturated ideals, and non-regular ultrafilters.
  part {I}.
\newblock {\em Annals of Mathematics}, 127(1):1--47, 1988.

\bibitem{MR2860182}
Sy-David Friedman and Peter Holy.
\newblock Condensation and large cardinals.
\newblock {\em Fund. Math.}, 215(2):133--166, 2011.

\bibitem{MR3235820}
Sy-David Friedman, Tapani Hyttinen, and Vadim Kulikov.
\newblock Generalized descriptive set theory and classification theory.
\newblock {\em Memoirs of the American Mathematical Society}, 230(1081):vi+80,
  2014.

\bibitem{FriedmanWu}
Sy-David Friedman and Liuzhen Wu.
\newblock Large cardinals and {$\Delta_1$}-definablity of the nonstationary
  ideal.
\newblock Preprint.

\bibitem{MR3320477}
Sy-David Friedman, Liuzhen Wu, and Lyubomyr Zdomskyy.
\newblock {$\Delta_1$}-definability of the non-stationary ideal at successor
  cardinals.
\newblock {\em Fundamenta Mathematicae}, 229(3):231--254, 2015.

\bibitem{fuchs_2018}
Gunter Fuchs.
\newblock Hierarchies of forcing axioms, the continuum hypothesis and square
  principles.
\newblock {\em The Journal of Symbolic Logic}, 83(1):256–282, 2018.

\bibitem{fuchs2020canonical}
Gunter Fuchs.
\newblock Canonical fragments of the strong reflection principle.
\newblock {\em J. Math. Log.}, 21(3):Paper No. 2150023, 58, 2021.

\bibitem{MR1877015}
Tapani Hyttinen and Mika Rautila.
\newblock The canary tree revisited.
\newblock {\em The Journal of Symbolic Logic}, 66(4):1677--1694, 2001.

\bibitem{JECH1973165}
Thomas~J. Jech.
\newblock Some combinatorial problems concerning uncountable cardinals.
\newblock {\em Annals of Mathematical Logic}, 5(3):165 -- 198, 1973.

\bibitem{MR1940513}
Thomas~J. Jech.
\newblock {\em Set theory}.
\newblock Springer Monographs in Mathematics. Springer-Verlag, Berlin, 2003.
\newblock The third millennium edition, revised and expanded.

\bibitem{MR2840749}
Ronald Jensen.
\newblock Subcomplete forcing and $\mathcal{L}$-forcing.
\newblock In {\em {$E$}-recursion, forcing and {$C^*$}-algebras}, volume~27 of
  {\em Lect. Notes Ser. Inst. Math. Sci. Natl. Univ. Singap.}, pages 83--182.
  World Sci. Publ., Hackensack, NJ, 2014.

\bibitem{doi:10.1002/malq.200410101}
Bernhard K{\"o}nig and Yasuo Yoshinobu.
\newblock Fragments of martin's maximum in generic extensions.
\newblock {\em Mathematical Logic Quarterly}, 50(3):297--302, 2004.

\bibitem{MR1782117}
Paul Larson.
\newblock Separating stationary reflection principles.
\newblock {\em The Journal of Symbolic Logic}, 65(1):247--258, 2000.

\bibitem{LarsonSchindlerWu}
Paul Larson, Ralf Schindler, and Liuzhen Wu.
\newblock {$\mathrm{NS}_{\omega_1}$ is not $\Pi_1$-definable}.
\newblock Handwritten notes.

\bibitem{MR2987148}
Philipp L\"{u}cke.
\newblock {$\Sigma^1_1$}-definability at uncountable regular cardinals.
\newblock {\em The Journal of Symbolic Logic}, 77(3):1011--1046, 2012.

\bibitem{MR3952233}
Philipp L\"{u}cke.
\newblock Closed maximality principles and generalized {B}aire spaces.
\newblock {\em Notre Dame Journal of Formal Logic}, 60(2):253--282, 2019.

\bibitem{MR683153}
Menachem Magidor.
\newblock Reflecting stationary sets.
\newblock {\em The Journal of Symbolic Logic}, 47(4):755--771 (1983), 1982.

\bibitem{MR1222536}
Alan~H. Mekler and Saharon Shelah.
\newblock The canary tree.
\newblock {\em Canad. Math. Bull.}, 36(2):209--215, 1993.

\bibitem{MR1242054}
Alan~H. Mekler and Jouko V\"{a}\"{a}n\"{a}nen.
\newblock Trees and {$\Pi^1_1$}-subsets of {$^{\omega_1}\omega_1$}.
\newblock {\em The Journal of Symbolic Logic}, 58(3):1052--1070, 1993.

\bibitem{MR1623206}
Saharon Shelah.
\newblock {\em Proper and improper forcing}.
\newblock Perspectives in Mathematical Logic. Springer-Verlag, Berlin, second
  edition, 1998.

\bibitem{MR2723878}
W.~Hugh Woodin.
\newblock {\em The axiom of determinacy, forcing axioms, and the nonstationary
  ideal}, volume~1 of {\em De Gruyter Series in Logic and its Applications}.
\newblock Walter de Gruyter GmbH \& Co. KG, Berlin, revised edition, 2010.

\end{thebibliography}

\end{document}